\documentclass[final,12pt]{colt2022} 

\title[Nonparametric Estimation of Interacting Particle Systems]{Mean-Field Nonparametric Estimation of Interacting Particle Systems}
\usepackage{times}
\usepackage{comment}
\usepackage{amssymb}
\usepackage{amsmath}
\usepackage{enumitem}

 \coltauthor{\Name{Rentian Yao} \Email{rentian2@illinois.edu}\\
  \Name{Xiaohui Chen} \Email{xhchen@illinois.edu}\\
  \Name{Yun Yang} \Email{yy84@illinois.edu}\\
  \addr University of Illinois at Urbana-Champaign}

\newcommand{\norm}[1]{{\left\vert\kern-0.25ex\left\vert #1 
    \right\vert\kern-0.25ex\right\vert}}
\newcommand{\wht}{\widehat}

\newcommand{\dd}{{\rm d}}

\newcommand{\argmax}{\mathop{\rm argmax~}}
\newcommand{\barP}{\overline{\mb P}^N}
\newcommand{\Lip}{\textrm{Lip}}
\newcommand{\ri}{\textrm{(i)}}
\newcommand{\rii}{\textrm{(ii)}}

\def\mb{\mathbb}
\def\m{\mathcal}

\newtheorem{assumption}{Assumption}

\begin{document}

\maketitle

\begin{abstract}%
  This paper concerns the nonparametric estimation problem of the distribution-state dependent drift vector field in an interacting $N$-particle system. Observing single-trajectory data for each particle, we derive the mean-field rate of convergence for the maximum likelihood estimator (MLE), which depends on both Gaussian complexity and Rademacher complexity of the function class. In particular, when the function class contains $d$-variate $\alpha$-H{\"o}lder smooth  functions, our rate of convergence is minimax optimal on the order of $N^{-\frac{\alpha}{d+2\alpha}}$. Combining with a Fourier analytical deconvolution argument, we derive the consistency of MLE for the external force and interaction kernel in the McKean-Vlasov equation.
\end{abstract}


\section{Introduction}
Recent years have seen increasing research interest and progress in learning dynamical pattern of a large interacting particle system (IPS). Motivating applications on modeling collective behaviors come from statistical physics~\citep{PhysRevLett.96.104302}, mathematical biology~\citep{Mogilner:1999aa,Topaz:2006aa}, social science~\citep{MotschTadmor2014}, stochastic control~\citep{BuckdahnLiMa2017}, mean-field games~\citep{CarmonaDelarue2018_meanfieldgamsI}, and more recently computational statistics on high-dimensional sampling~\citep{NIPS2017_17ed8abe,LuLuNolen2019} and machine learning for neural networks~\citep{MeiE7665,MeiMisiakiewiczMontanari2019_colt,ChizatBach2018_nips,SIRIGNANO20201820,SirignanoSpiliopoulos2020}. Due to the large number of particles with interactions, such dynamical systems are high-dimensional and often non-linear. In this paper, we consider a general interacting $N$-particle system described by the stochastic differential equations (SDEs)
\begin{align}\label{eqn: interact_system}
\dd X_t^i = b^*(t, \mu_t^N, X_t^i)\,\dd t + \sigma^*(t, X_t^i) \,\dd W_t^i,\quad 1\leq i\leq N,
\end{align}
where $(W_t^1)_{t\geq0},\dots,(W_t^N)_{t\geq0}$ are independent Brownian motions on the $d$-dimensional Euclidean space $\mb R^d$, $\mu_t^N = N^{-1} \sum_{i=1}^N \delta_{X_t^i}$ is the empirical law of the interacting particles, and the initialization $X^1_0,\dots,X^N_0$ are i.i.d.~$\mb R^d$-valued random variables with a common law $\mu_0$, independent of $(W_t^i)_{t\geq0}$. 
Here in the non-linear diffusion process~\eqref{eqn: interact_system}, letting $\m P(\mb R^d)$ be the space of all probability measures on $\mb R^d$, the vector field $b^* : \mathbb{R}_+ \times \mathcal{P}(\mb R^d) \times \mb R^d \to \mathbb{R}^d$ is a distribution-state dependent drift vector field to be estimated and $\sigma^*$ is a known diffusion function (or volatility coefficient) quantifying the magnitude of the self-energy of the particle.
For simplicity, we focus on systems with time-homogeneous and space-(one-)periodic drift vector field $b^*(t,\nu,x) =: b^*(\nu,x)$ satisfying $b^*(\nu,x+m) = b^*(\nu,x)$ for every $m\in\mb Z^d$, and constant diffusion function $\sigma^*(t, x) \equiv 1$. The periodic model effectively confines the SDEs to a compact state space as the $d$-dimensional torus $\mb T^d$, and is commonly adopted in the SDE analysis to avoid boundary issues~\citep{van2016gaussian,pokern2013posterior,nickl2020nonparametric}.
Suppose that we observe continuous-time single-trajectory data for each particle $\mathcal{X}_T = \{(X_t^1, \dots, X_t^N):0\leq t\leq T\}$ in a finite time horizon $T > 0$. Our goal is to derive a statistically valid procedure to estimate the vector field $b^*$ in a large IPS based on the data $\mathcal{X}_T$.

\subsection{System governed by external-interaction force}
In the periodic setting, the values of the process $(X_t^i)$ modulo $\mb Z^d$ contain all relevant
statistical information, so we can identify the law of $(X_t^i)$ with a uniquely defined probability measure on $\mb T^d$ (cf.~Section 2.2 in~\cite{nickl2020nonparametric} for further details).
Under such identification, one important class of IPS with a time-homogeneous drift vector field can be represented as
\begin{align}\label{eqn:external-interaction_force}
    b^\ast(\nu, x) = \int_{\mb T^d}\tilde{b}^\ast(x, y)\,\dd\nu(y),\quad\mbox{with}\quad
    \tilde{b}^\ast(x,y) = G^\ast(x) + F^\ast(x-y),
\end{align}
for $\nu \in \mathcal{P}(\mb T^d)$ and continuous $F^\ast, G^\ast : \mb T^d \to \mb R^d$. In this case, one can interpret $G^\ast$ as an external force to the global system characterizing the drift tendency of particles and $F^\ast$ as an interaction kernel between particles. Then the IPS in~\eqref{eqn: interact_system} can be reformulated as
\vspace{-0.5em}
\begin{align*}
\dd X_t^i = G^\ast(X_t^i) \,\dd t + \frac{1}{N}\sum_{j=1}^N F^\ast(X_t^i - X_t^j) \, \dd t + \dd W_t^i.\\[-5ex]
\end{align*}
In statistical mechanics, microscopic behaviors of $N$ random particles are usually related to explain some observed macroscopic physical quantities (e.g., temperature distributions) in the sense that the evolution of the empirical law $\mu_t^N$ of the particles converges to a non-random mean-field limit $\mu_t$ as $N \to \infty$ and the probability measure flow $\mu_t(x) := \mu(x,t)$ solves the \emph{McKean-Vlasov equation}~\citep{McKean1966}
\begin{equation}
    \label{eqn:mckean-vlasov_eqn}
    \partial_t \mu = \Delta \mu + \mbox{div}\bigg(\mu \, \bigg[G^\ast + \int_{{\mb T}^d} F^\ast(\cdot - y) \mu_t(\dd y) \bigg] \bigg),
\end{equation}
which is a non-linear parabolic partial differential equation (PDE). For this special class of IPS, a further goal is to study the identifiability of $(F^\ast,G^\ast)$ and consistency of the derived estimators.

\subsection{Related work}

It is a classical result that $N$-particle interacting system (\ref{eqn: interact_system}) admits a unique strong solution, when both $b^*$ and $\sigma^*$ are Lipschitz continuous and the solution converges to its mean-field limit McKean-Vlasov stochastic differential equation (MVSDE) both in pathwise and weakly under the same Lipschitz condition~\citep{carmona2016lectures,CarmonaDelarue2018_meanfieldgamsI}. The latter is usually known as propagation of chaos~\citep{Sznitman1991}. Another inspiring work from \cite{lacker2018strong} showed that the convergence can be proved in a much stronger topology ($\tau$-topology), when volatility coefficient $\sigma^*$ involves no interaction term. 

Several works about learning the interaction kernel of interacting particle system have be done lately. \cite{bongini2017inferring} proposed an estimator by minimizing the discrete error functional, whose convergence rate is usually no faster than $N^{-1/d}$. This reflects the phenomenon of curse-of-dimension. \cite{lu2019nonparametric} constructed the least square estimator for interaction kernel, which enjoys an optimal rate of convergence under mild conditions. These two works were done under a noiseless setting, i.e., the system evolves according to an ordinary differential equation and initial conditions of agents are i.i.d. As for the stochastic system, \cite{li2021identifiability} studied the learnability (identifiability) of interaction kernel by maximum likelihood estimator (MLE) under the coercivity condition, and \cite{lang2021identifiability} provided a complete characterization of learnability. \cite{della2021nonparametric} investigated a nonparametric estimation of the drift coefficient, and the interaction kernel can be separated by applying Fourier transform for deconvolution. The convergence result is provided under a fixed time horizon, meaning that time $T$ is fixed in their asymptotic result. Another nonparametric estimation algorithm based on least squares was proposed by \cite{lang2020learning}. 

Estimating parameters of interacting systems by maximum likelihood can date back to 1990. \cite{kasonga1990maximum} proved the consistency and asymptotic normality of MLE for linear parametrized interacting systems. As for MVSDE, \cite{wen2016maximum} discussed the consistency of MLE in a broad class of MVSDE, based on a single trajectory $(x_t)_{0\leq t\leq T}$. \cite{liu2020parameter} extended it to path-dependent case with non-Lipschitz coefficients. Both of these works focused on the asymptotic behaviour when $T\to\infty$. \cite{sharrock2021parameter} studied the case with $N$ realizations of MVSDE, and the case of $N$ interacting particle systems, under which consistency of MLE was proved when $N\to\infty$ and an online parameter estimation method was also discussed. \cite{chen2021maximum} showed that MLE has optimal rate of convergence in mean-field limit and long-time dynamics, when assuming linear interactions and no external force.

\vspace{-1em}
\subsection{Our contributions}
We provide a rigorous non-asymptotic analysis of MLE of drift coefficient restricted on a general class of functions with certain smoothness condition. \cite{della2021nonparametric} proposed a kernel based estimation procedure for the same estimation problem. However, unlike our method, the behaviour of estimation based on kernel method rely heavily on tuning the bandwidth and their analysis does not involve uniform laws of dependent variables. Moreover, the MLE framework provides a unified and principled strategy that naturally incorporates finer structures such as~\eqref{eqn:external-interaction_force} in modelling the drift vector field $b^\ast$. In comparison, the kernel method requires further specialized steps for separating interaction force $F^\ast$ from the external force $G^\ast$ after the estimation of $b^\ast$. As a consequence, we do not need to explicitly specify the deconvolution operator ($\m L$ in Assumption~\ref{assump: non-stationary}) and only need to assume its existence in our consistency analysis, making the MLE approach more robust to changes in problem characteristics and less sensitive to parameter tuning. 

In our study, there are several obstacles while analyzing the MLE, some of which make our analysis technically more involved than that for the kernel method. Firstly, observations in $\m X_T$ are not i.i.d.~because of interaction among particles from the drift $b^\ast(\mu_t^N,\cdot)$. To decouple the dependence, we follow~\cite{della2021nonparametric} by using Girsanov's theorem to construct a new measure, under which the trajectory of particles becomes i.i.d. However, this change of measure will introduce some additional decoupling errors in our analysis of the MLE that is not present in the analysis of the kernel method~\citep{della2021nonparametric}. Dealing with these decoupling errors requires substantial efforts and is technically highly non-trivial.
Secondly, we derive a new and specialized maximal inequality (cf.~Lemma~\ref{lem: order_of_ito}) for handling the supreme of an unbounded process involving the It\^o integral that appears in our analysis. The derived maximal inequality is general and interesting in its own right, and can be applied to other problems involving diffusion processes beyond our current setting.
Thirdly, a standard union bound argument cannot be applied to deal with the decoupling error (see the discussion after equation~\eqref{eqn: new_basic_ineq} for a precise meaning) in terms of
the supreme of a random process expressed as the average of correlated It\^o integrals that naturally appears when analyzing the MLE. To address this issue, we develop a concentration inequality for U-statistics involving It\^o integrals (cf.~Lemma~\ref{lem: U_statistics}), which is then combined with chaining and leads to a new maximal inequality for U-processes (cf.~Lemma~\ref{lem: decoupling_err_Ito}). This refined maximal inequality helps us derive a better rate in our problem than using existing general versions of the inequality.

\subsection{Notation}
Let $\mb Z$ ($\mb N$) denote the set of all (non-negative) integers.
For any arbitrary functions $f: \mb T^d\to \mb R^d$, the Fourier series $(f)_k$ of $f$ is defined as
\begin{align*}
    (f_i)_k = \int_{\mb T^d}f_i(x) e^{-2\pi ik\cdot x}\,\dd x,\quad 1\leq i\leq d, \,k\in\mb Z^d,
\end{align*} 
where we let $f = (f_1,\cdots, f_d)^T$ and $(f)_k = ((f_1)_k, \cdots, (f_d)_k)^T$ are $d$-dimensional column vectors. Properties of Fourier analysis on torus can be found in Chapter 3 of \cite{grafakos2008classical}. 

 For $k = (k_1, \cdots, k_d)^T\in\mb Z^d$, let $|k| = k_1 + \cdots + k_d$ be the sum of all elements of $k$, and $D^k = \partial_{k_1\cdots k_d}$ is a $|k|$-th order partial derivative. We use $\norm{\cdot}$ for $l_2$-norm of a vector, and $\norm{\cdot}_2$ for $L^2(\mb T^d)$-norm of a (vector-valued) function, i.e., $\norm{f}_2^2 = \int_{\mb T^d}\norm{f(x)}^2\,\dd x.$ For a Lipschitz function $f$, we denote $\|f\|_{\text{Lip}}$ is the smallest constant $C > 0$ such that $\|f(x)-f(y)\| \leq \|x-y\|$ for all $x, y \in \mb T^d$.
Let $\norm{\cdot}_{H^1}$ be the Sobolev norm defined as $\norm{f}_{H^1}^2 = \norm{f}_2^2 + \sum_{i=1}^d \norm{\nabla f_i}_2^2.$
In addition, for a function $b(\nu, x)$, we define seminorms $\norm{b}_E^2 := \int_0^T\!\!\!\int_{\mb T^d}\norm{b(\mu_t, x)}^2\,\dd\mu_t(x)\dd t,$
and $\norm{b}_X^2 := N^{-1} \sum_{i=1}^N\int_0^T\norm{b(\mu_t, X_t^i)}^2\,\dd t$, and let the norm $|\!|\tilde{\cdot}|\!|_\Lip$ of any $b(\nu, \cdot) = \int_{\mb T^d}\tilde{b}(\cdot, y)\,\dd\nu(y)$ be $|\!|\tilde{b}|\!|_\Lip$.

For $0 < \beta < \infty$, let $\psi_{\beta}$ be the function on $[0,\infty)$ defined by $\psi_{\beta} (x) = e^{x^{\beta}}-1$, and for a real-valued random variable $\xi$,  define $\| \xi \|_{\psi_\beta}=\inf \{ C>0: \mathbb{E}[ \psi_{\beta}( | \xi | /C)] \leq 1\}$.
For $\beta \in [1,\infty)$, $\|\cdot\|_{\psi_{\beta}}$ is an \emph{Orlicz norm}, while for $\beta \in (0,1)$, $\| \cdot \|_{\psi_{\beta}}$ is not a norm but a quasi-norm, i.e., there exists a constant $C_{\beta}$ depending only on $\beta$ such that $\| \xi_{1} + \xi_{2} \|_{\psi_{\beta}} \leq C_{\beta} ( \| \xi_{1} \|_{\psi_{\beta}} + \| \xi_{2} \|_{\psi_{\beta}})$. Indeed, there is a norm equivalent to $\| \cdot \|_{\psi_{\beta}}$ obtained by linearizing $\psi_{\beta}$ in a neighborhood of the origin; cf. Lemma C.2 in~\cite{chen2019randomized}. 

For a function class $\m H$, define the shifted class $\m H^* := \m H - h^*$ for some $h^* \in \m H$. The function class $\m H^*$ is \emph{star-shaped} (or equivalently $\m H$ is star-shaped around $h^*$) if for any $h \in \m H$ and $\alpha \in [0,1]$, the function $\alpha h \in \m H^*$; cf. Chapter 13 of~\cite{wainwright2019high}. We use $N(\varepsilon, \m H, \|\cdot\|)$ to denote the $\varepsilon$-covering number for the function class $\m H$ under the metric induced by the norm $\|\cdot\|$.

\section{Constrained Maximum Likelihood Estimation}

Let $\m C = \m C([0, T], (\mb T^d)^N)$ be the set of all continuous functions on $(\mb T^d)^N$, and $\{\m F_t: 0\leq t\leq T\}$ be the filtration generated by our observation $\mathcal{X}_T$. According to Girsanov's theorem (Theorem 1.12 in \cite{kutoyants2004statistical}), the log-likelihood ratio function for the continuous time trajectory data $\m X_T$ takes the form as
\begin{equation}\label{eqn:radon-nikodyn_derivative_p}
\begin{aligned}
    L_T(b) :\,=\log \frac{\,\dd {\mb P}^N_{b}}{\,\dd {\mb P}_0^N}=\sum_{i=1}^N\int_0^T\!\!\!\big\langle b(\mu_t^N, X_t^i), \dd X_t^i\big\rangle 
   - \frac{1}{2}\sum_{i=1}^N\int_0^T \!\norm{b(\mu_t^N, X_t^i)}^2\,\dd t,
\end{aligned}  
\end{equation}
where $\frac{\,\dd {\mb P}^N_{b}}{\,\dd {\mb P}_0^N}$ denotes the Radon-Nikodym derivative of the probability measure ${\mb P}^N_{b}$ associated with $\m X_T$ from model $\dd X_t^i = b(t, \mu_t^N, X_t^i)\,\dd t + \dd W_t^i$, $1\leq i\leq N$, relative to the base measure ${\mb P}_0^N$.

When the drift vector field $b^*$ is driven by the external-interaction force in~\eqref{eqn:external-interaction_force}, it is natural to consider the maximum likelihood estimator (MLE) for $b^*$ in the function class
\begin{align*}
    \m H = \bigg\{b: \m P(\mb T^d)\times \mb T^d \to \mb R^d \,\Big|\,\exists\, F, G\in\tilde{\m H}, b(\nu, x) = G(x) + \int_{\mb T^d} F(x-y)\,\dd \nu(y)\bigg\},
\end{align*}
where $\tilde{\m H}$ is a uniformly bounded function class whose elements map from $\mb T^d$ to $\mb R^d$ with certain smoothness (cf. assumptions in Theorem~\ref{thm: converge_of_b} below). Note that for $b \in \m H$, we can equivalently compute the MLE
$\wht b_N := \argmax_{b\in\m H} L_T(b)$
by first obtaining the MLE of $F^*$ and $G^*$ as in~\eqref{eqn:external-interaction_force}
\begin{align}
    (\wht F_N, \wht G_N) &= \argmax_{F, G\in\tilde{\m H}} \tilde{L}_T(F, G), \quad \mbox{subject to} \ \ \int_{\mb T^d} F(x)\,\dd x=0, \label{eqn: MLE_FG}\\
  \mbox{where} \qquad  \tilde{L}_T(F, G) &= -\frac{1}{2}\sum_{i=1}^N\int_0^T\Big\|G(X_t^i) + \frac{1}{N}\sum_{j=1}^N F(X_t^i - X_t^j)\Big\|^2\,\dd t \notag\\
    &\qquad\qquad\qquad + \sum_{i=1}^N\int_0^T\Big\langle G(X_t^i) + \frac{1}{N}\sum_{j=1}^N F(X_t^i-X_t^j),\,\dd X_t^i\Big\rangle, \notag
\end{align}
and then setting 
\begin{align*}
    \wht b_N(\nu, x) = \wht G_N(x) + \int_{\mb T^d}\wht F_N(x-y)\,\dd \nu(y),\quad\forall\,\nu\in\m P(\mb T^d).
\end{align*}
Note that for any solution $(\wht F_N, \wht G_N)$ of~\eqref{eqn: MLE_FG} and a constant $C\neq 0$, $(\wht F_N + C, \wht G_N - C)$ is also a solution. Therefore, we impose an additional restriction $\int_{\mb T^d} F^\ast(x)\,\dd x = 0$ for the sake of identifiability of the interaction kernel. This also explains the extra constraint  $\int_{\mb T^d} F(x)\,\dd x=0$ imposed in the estimation procedure~\eqref{eqn: MLE_FG}.

\section{Convergence Rate of Constrained MLE}

In this section, we first derive a general rate of convergence for the constrained MLE $\hat{b}_N$ based on the entropy method which will lead to a computable and simple bound when specialized to $\alpha$-smooth H\"older function class. In the latter case, we will show that the constrained MLE achieves the minimax optimal rate in Section~\ref{subsec:rate_smooth_holder_class}. As a consequence, we derive the consistency $\wht G_N$ and $\wht F_N$ in Section~\ref{subsec:consistency_G+F} for the external-interaction force model~\eqref{eqn:external-interaction_force}.

\begin{assumption}\label{assump: pointwise_measurable}
The class $\tilde{\m H}$ is \emph{pointwise measurable}, i.e., it contains a countable subset $\m G$ such that for every $h \in \m H$ there exists a sequence $g_m \in \m G$ such that $g_m \to h$ pointwise.
\end{assumption}
Assumption~\ref{assump: pointwise_measurable} is made to avoid measurability issues~\citep{wellner2013weak} since it guarantees that the supremum of a suitable empirical process indexed by $\tilde{\m H}$ is a measurable map. Define the \emph{localized} function class
\[
{\m H}_u^* = \{f \in {\m H}^* : \|f\|_E \leq u \}.
\]
If $B = \sup_{f\in\m H_u^\ast}\|f\|_\infty < \infty$, let $H(\varepsilon, \m H^\ast_u)$ be the cardinality of the smallest set $S\subset \m H^\ast_u$ such that $\forall\,g\in\m H_u^\ast$ there exists $f\in S$ satisfying $\norm{f-g}_E \leq \varepsilon u$ and $\norm{f-g}_\infty \leq \varepsilon B$. We shall notice that $H(\varepsilon, \m H^\ast)$ requires covering w.r.t. both $\norm{\cdot}_\infty$ and $\norm{\cdot}_E$, while $N(\varepsilon, \m H^\ast, \norm{\cdot})$ only requires covering w.r.t. a certain norm $\norm{\cdot}$. We further define several entropy integrals that control the complexity of our nonparametric estimation problem:
\begin{align*}
    J_1(u) &:= \sqrt{T}B\int_0^\frac{1}{2}\log\big(1 + H(\varepsilon, \m H^\ast_u)\big)\,\dd \varepsilon + u\int_0^\frac{1}{2}\sqrt{\log\big(1 + H(\varepsilon,\m H^\ast_u)\big)}\,\dd \varepsilon, \\
    J_2(L) &:= \int_0^{\frac{L}{2}}\log\big(1 + N(\varepsilon, \m H^\ast, \norm{\tilde{\cdot}}_{\textrm{Lip}})\big)\,\dd\varepsilon, \quad J_3(B) := \int_0^\frac{B}{2}\sqrt{\log\big(1 + N(u, \m H^\ast, \norm{\cdot}_\infty)\big)}\,\dd u, \\
    J_4(B) &:= \int_0^\frac{B}{2} \big[\log(1 + N(u,\m H^\ast, \norm{\cdot}_\infty))\big]^{\frac{3}{2}}\,\dd u, \quad J_5(r) := \int_0^\frac{r}{2}\sqrt{\log N(s/\sqrt{T}, \m H^\ast, \norm{\cdot}_\infty)}\,\dd s.
\end{align*}

We assume that $J_1(u), J_2(L), J_3(B), J_4(B), J_5(r)$ are finite for some function class parameters on $\m H^*$ and some localization parameter $u$.

\begin{theorem}[Rate of convergence of constrained MLE]\label{thm: converge_of_b}
Suppose the function class $\tilde{\m H}$ satisfies Assumption~\ref{assump: pointwise_measurable} such that $\norm{F}_\infty \leq B$ and $\norm{F}_\Lip \leq L$ for all $F\in\tilde{\m H}$. Assume there exist positive constants $\delta_N$ and $r_N$ satisfying
\begin{align*}
    \mb E_{\barP}\sup_{g\in\m H^\ast_{\delta_N}}\bigg|\frac{1}{N}\sum_{i=1}^N\int_0^T\big\langle g(\mu_t, X_t^i),\,\dd\overline{W}_t^i\big\rangle\bigg| \leq 2\delta_N^2 \quad \mbox{and} \quad J_5(r_N) \leq \frac{\sqrt{N}r_N^2}{6CB\sqrt{T}},
\end{align*}
where $(\overline{W}_t^i)_{t\geq0}$ is the (transformed) Brownian motion defined in~\eqref{eqn:transformed_BM} and $\barP$ is the associated probability measure. If $b^\ast\in\m H$, then
\begin{align}\label{eqn: converge_of_b}
    \big|\!\big|\wht b_N - b^\ast\big|\!\big|_E \leq 48\Big(\delta_N + r_N + \sqrt{\frac{\log N}{N}}\Big)
\end{align}
with probability at least
\begin{align}\label{eqn: tail_prob}
\begin{aligned}
    & 1 - \bigg[\kappa_1\exp\bigg\{-\frac{\kappa_2N\delta_N^2}{3}\bigg\} + 3\kappa_1\exp\bigg\{-\frac{\kappa_2N\delta_N^2}{C\log NJ_1(1)}\bigg\}\bigg] - 2\kappa_1\exp\bigg\{-\frac{\kappa_2\log N}{4TK_1|\!|\tilde{b}^\ast|\!|_{\textrm{Lip}}^2}\bigg\}\\
    &\quad- 2\kappa_1\exp\bigg\{-\frac{\kappa_2\log N}{CLK_1J_2(L)}\bigg\} - 2\kappa_1\exp\bigg\{-\frac{\kappa_2N(\log N)^2}{CJ_3^2(B)}\bigg\}\\
    &\quad - 2\kappa_1\exp\bigg\{-\Big(\frac{\kappa_2\log N}{CJ_4(B)\log\log N}\Big)^{\frac{2}{3}}\bigg\} - \kappa_1\exp\bigg\{-\frac{\kappa_2Nr_N^2}{36C^2B^2T}\bigg\}.
\end{aligned}
\end{align}
Here $\kappa_1, \kappa_2, K_1, C$ are some positive constants.
\end{theorem}

There are some interesting remarks for Theorem~\ref{thm: converge_of_b} in order.

\begin{enumerate}[nolistsep]
    \item Theorem~\ref{thm: converge_of_b} remains true for drift vector field $b^*$ with anisotropic interaction force, namely $b^\ast(\nu, x) = \int_{\mb T^d}\tilde{b}^\ast(x, y)\,\dd\nu(y)$ with  $\tilde{b}^\ast = G^*(x) + F^*(x, y)$. In this case, we  need to consider functions in $\tilde{\m H}$ map from $\mb T^{2d}$ to $\mb R^d$.
    
    \item $\delta_N$ corresponds to the Gaussian complexity of function class $\m H^\ast$ in a discrete setting, while $r_N$ is an upper bound of Rademacher complexity of $\m H^\ast$. Intuitively, the estimation problem will be harder if the function class is more complex, and thus $\delta_N$ and $r_N$ should affect the rate of convergence in certain extent. It is quite common that the rate of convergence of nonparametric estimation depends on both $\delta_N$ and $r_N$, see e.g. Corollary 14.15 in \cite{wainwright2019high} as an example. 
    
    \item Usually we use chaining method~\citep{wainwright2019high, geer2000empirical} to bound the expectation term in order to derive an explicit form of $\delta_N$. Since we can show the sum of i.i.d.~It\^o integral is sub-exponential, one direct method is to use $\psi_1$-norm to bound the expectation (see Lemma \ref{lem: psi1_chaining}). Another possible way to apply chaining is based on Bernstein-Orlicz norm~\citep{van2013bernstein}. In this case, bracketing number rather than covering number will be needed.
    
    \item In practice, we can only observe discretely sampled trajectory data. Thus it is important to understand the impact of discretization. Note that the drift coefficient of an SDE is related to the mean function of the corresponding stochastic process. It is known that, in functional data analysis, convergence rate of estimation of mean function will not be affected if sampling frequency is large enough~\citep{cai2011optimal}. This is called phase \emph{transition phenomenon}. We conjecture that a similar phase transition phenomenon may occur in our setting. In parametric setting, if parameters have linear effects on the interaction function, the problem is studied by~\cite{bishwal2011estimation} and the convergence rate remains optimal as $O(n^{-1/2})$. A rigorous analysis in nonparametric setting is open and we leave it as the future work.
    
    \item As is common in the nonparametric regression literature, the constrained MLE is usually adopted due to its technical simplicity; while the penalized MLE (e.g. by adding a squared RKHS norm penalty) is used for practical computation since they are dual to each other: with proper choice of the regularization parameter, they lead to the same solution. With discretely sampled trajectory data, the penalized MLE can be equivalently formulated as a finite dimensional optimization problem due to representer’s theorem (kernel trick).
\end{enumerate}

\noindent Our proof of Theorem~\ref{thm: converge_of_b} is quite involved. The proof builds on a number of recently developed technical tools such as change of measure via Girsanov's theorem for decoupling the IPS, concentration inequalities for unbounded empirical processes and degenerate $U$-processes, localization technique for sum of i.i.d.~It\^o integrals. We shall sketch the main structure of the argument in Section~\ref{subsec:proof_sketch_rate_constrainedMLE} and defer the complete proof details to Appendix. 

\subsection{Application to H\"older smooth function class}
\label{subsec:rate_smooth_holder_class}
Consider the case where $\tilde{\m H} = C_1^\alpha\big([0, 1]^d\big)$ with fixed smooth parameter $\alpha$ and bounded $\alpha$-H\"{o}lder norm. Here, the $\alpha$-smooth H{\"o}lder function class is the set of all functions $f$ such that 
\begin{align*}
    \norm{f}_\alpha := \max_{|k|\leq \lfloor\alpha\rfloor}\sup_{x}\Big|D^kf(x)\Big| + \max_{|k| = \lfloor\alpha\rfloor}\sup_{x\neq y}\frac{\big|D^{k}f(x) - D^{k}f(y)\big|}{\norm{x- y}^{\alpha - \lfloor\alpha\rfloor}} \leq M
\end{align*}
with some fixed $M > 0$. In this case, assumptions of boundedness in Theorem~\ref{thm: converge_of_b} hold for $B = L = M$.

From Theorem 2.7.1 in \cite{wellner2013weak} we know
\begin{align*}
    \log N(\varepsilon, \tilde{\m H}, \norm{\cdot}_\infty) &\lesssim \varepsilon^{-\frac{d}{\alpha}}.
\end{align*}
By Theorem~\ref{thm: converge_of_b} we obtain the following rate of convergence in Corollary~\ref{coro: holder_smooth}. The detailed proof is deferred to Appendix. Our theory works for Sobolev (Besov) space with bounded norm as well, and an analogy of the following corollary can be derived by a similar argument. 

\begin{corollary}[H\"older smooth drift estimation error]\label{coro: holder_smooth}
Suppose the $\alpha$-H\"{o}lder smooth function class $\tilde{\m H}$ satisfies $\alpha > 3d/2$, then there are positive constants $C$ and $C'$ such that
\begin{align}\label{eqn:minimax_smooth_holder_class}
    \big|\!\big| \wht b_N - b^\ast\big|\!\big|_E \lesssim N^{-\frac{\alpha}{d + 2\alpha}}
\end{align}
with probability at least $1 - C\exp\big\{-C'(\frac{\log N}{\log\log N})^{2/3}\big\}$.
\end{corollary}
Note that rate of convergence for the MLE derived in~\eqref{eqn:minimax_smooth_holder_class} attains the minimax rate of estimating the drift term in IPS~\citep{della2021nonparametric}. When $\alpha \leq 3d/2$, the term $J_4(B)$ in our general result Theorem~\ref{thm: converge_of_b} (also cf. Lemma \ref{lem: decoupling_err_Ito} for definition) is not finite. However, we still can derive a rate of convergence for $\hat{b}_N$ by refining the definition of $J_4$ as the one in \cite{van2014uniform} to avoid the integrability issue. The price we pay is that the tail probability converges to zero more slowly, which would translate to a sub-optimal rate of convergence of $\wht b_N$.

\subsection{Estimating interaction kernel in Vlasov model}
\label{subsec:consistency_G+F}
In this section, we specialize our theory to the external-interaction force system, and study the convergence behaviour of interaction kernel $F^\ast$. This is an inverse problem and can be done in a similar argument as in \cite{della2021nonparametric} by Fourier transform. \cite{della2021nonparametric} \emph{explicitly} use Fourier transform to construct an estimator based on deconvolving a kernel density estimator for $\mu_t$ from an estimator for $F^*\ast \mu_t$, where $\ast$ denotes the function convolution operator. In comparison, we \emph{implicitly} use Fourier transform to derive a stability estimate for translating an error bound on $\wht b_N$ to that on $\wht F_N$ in the analysis. More specifically, recall that
\begin{align}\label{eqn: invariance}
    \wht b_N(\mu_t, x) - b^\ast(\mu_t, x)= \wht G_N(x) - G^\ast(x) + \int_{\mb T^d}\big(\wht F_N(x-y) - F^\ast(x-y)\big)\,\dd\mu_t(y).
\end{align}
Let $L^2([0,T])$ denote the space of all square-integrable functions on $[0,T]$ and view $\mu(x,t)=\mu_t(x)$ as a function of $(x,t)$.
For any linear operator $\m L:\,L^2([0,T]) \to \mb R$, $f\mapsto \m L f := \int_0^T f(t)w(t)\,\dd t$, where $w$ is a bounded measurable function on $[0, T]$ such that $\int_0^T w(t)\,\dd t = 0$, we obtain by applying $\m L$ to both sides of~\eqref{eqn: invariance},
\begin{align}\label{eqn: Lb = F*Lu}
    \m L \big[(\wht b_N-b^\ast)(\mu, x)\big]  = \m L\big[(\wht F_N- F^\ast)\ast \mu(x)\big] = \big((\wht F_N - F^\ast)\ast \m L\mu\big)(x),
\end{align}
where we have used the property that $\m Lg=0$ for any $t$-independent function $g$.
Since the goal is to relate $\|\wht F_N - F^\ast\|_2$ to $\|\wht b_N - b^\ast\|_E$, we may apply Fourier transform to both sides of~\eqref{eqn: Lb = F*Lu} to deconvolute $\wht F_N- b^\ast$ and $\m L\mu$, leading to 
\begin{align}\label{eqn: deconvolution}
    \Big(\m L \big[(\wht b_N-b^\ast)(\mu, \cdot)\big]\Big)_k = (\wht F_N - F^\ast)_k\cdot (\m L\mu)_k,\quad \forall\,k\in\mb Z^d.
\end{align}
Note that by the definition of $\m L$, we have $(\m L\mu)_0 = \m L \int_{\mb T^d} \,\dd\mu(x) = \m L 1=0$, so equation~\eqref{eqn: deconvolution} only determines the Fourier coefficient of $(\wht F_N - F^\ast)_k$ for $k\neq 0$. However, $(\wht F_N - F^\ast)_0$ can be uniquely determined by our additional identifiability constraint $\int_{\mb T^d}\wht F_N(x)\,\dd x = \int_{\mb T^d}F^\ast(x)\,\dd x=0$.

To ensure that $\wht F_N - F^\ast$ remains small when $\wht b_N$ is close to $b^\ast$,  we need the following assumption motivated by identity~\eqref{eqn: deconvolution}.
\begin{assumption}\label{assump: non-stationary}
There exists a bounded measurable function $w(t)$ on $[0, T]$ such that $\int_0^T w(t)\,\dd t =0$ and $(\m L \mu)_k=\int_0^T(\mu_t)_k\,w(t)\,\dd t \neq 0$ for all nonzero $k\in\mb Z^d$.
\end{assumption}

Assumption~\ref{assump: non-stationary} guarantees the identifiability of interaction force $F^\ast$ (and therefore external force $G^\ast$) from the drift vector field $b^\ast$, and requires the system to be away from stationarity. To see this, consider the ideal setting where we exactly know the true $b^\ast$ and $\{\mu_t: 0\leq t\leq T\}$, and want to uniquely recover $F^\ast$ from~\eqref{eqn:external-interaction_force}. Suppose system~\eqref{eqn: interact_system} already attains stationarity, i.e. $\mu_t = \mu^\ast$ for all $t\in[0,T]$ with $\mu^\ast$ denoting the stationary distribution of the system which solves equation~\eqref{eqn:mckean-vlasov_eqn} with $\partial_t\mu=0$. Then it is impossible to separate out the time-homogeneous interaction term $F^\ast\ast \mu^\ast$ from the drift vector field
\begin{align*}
    b^\ast(\mu^\ast,x) = G^\ast(x) + \int_{\mb T^d} F^\ast(x-y)\,\dd\mu^\ast(y),
\end{align*}
since for any function $g:\,\mb T^d\to\mb R$, the new pair $(G',\,F') = (G^\ast-g\ast \mu^\ast,\, F^\ast+g)$ induces the same drift $b^\ast$.
However, this setting violates Assumption~\ref{assump: non-stationary} since $(\m L\mu)_k = 0$ for all $k\in\mb Z^d$.
In other words, the interaction force $F^\ast$ can only be recovered from the transient behaviour of the system, and Assumption~\ref{assump: non-stationary} is one mathematical description implying the system to be away from stationarity.


The linearly dependence structure~\eqref{eqn: deconvolution} in the frequency domain suggests that the estimation error of $\wht F_N$ depends on the accuracy of $\wht b_N$ and the behaviour of $(\m L \mu)_k$. It is common that an inverse problem related to a convolution equation, such as equation~\eqref{eqn: Lb = F*Lu}, tends to be numerically unstable, since it will be ill-posed when the Fourier coefficients $\{(\m L\mu)_k\}_{k=1}^\infty$ of $\m L\mu$ decay too fast~\citep{isakov2006inverse}. By quantifying the stability of solution to~\eqref{eqn: Lb = F*Lu} around the true interaction $F^\ast$, we arrive at the following corollary. 

\begin{corollary}[Interaction kernel estimation error]\label{cor: rate_of_interaction}
Let $\eta_N$ be the smallest integer satisfying $\eta_N(\delta_N + r_N + \frac{\log N}{N})\leq C_1 \inf_{0 < \norm{k} < \eta_N}\norm{(\m L\mu)_k}$. If $F^\ast$ and $G^\ast$ belong to $\tilde{\m H}$ and both have finite $H_1$-norms, then
\begin{align*}
    \big|\!\big|\wht F_N - F^\ast\big|\!\big|_2 \leq C_2 \,\eta_N^{-1}
\end{align*}
holds with at least probability given by~\eqref{eqn: tail_prob}. Here constants $(C_1,,C_2)$ are independent of $N$.
\end{corollary}
\begin{remark}
In the proof of Corollary~\ref{cor: rate_of_interaction}, we only used the condition that the Sobolev norm $\|F^\ast\|_{H^1}$ is finite. If we further use the condition that $F^\ast$ is $\alpha$-H\"{o}lder continuous with $\alpha>1$, then the error bound can be improved to $\eta_N^{-\alpha}$ where $\eta_N^\alpha(\delta_N + r_N + \frac{\log N}{N})\lesssim \inf_{0<\norm{k} < \eta_N}\norm{(\m L\mu)_k}$.
In particular, if Assumption~\ref{assump: non-stationary} holds, then Corollary~\ref{cor: rate_of_interaction} implies $\wht F_N$ to be a consistent estimator of $F^\ast$ as $N\to\infty$.
\end{remark}

\section{Proof of Theorem~\ref{thm: converge_of_b}}\label{subsec:proof_sketch_rate_constrainedMLE}

By decomposing the sample space into $\m E = \{\big|\!\big|\wht \Delta_N\big|\!\big|_E \leq \delta_N\}$ and $\m E^c$, Theorem~\ref{thm: converge_of_b} is true on the event $\m E$. So, we only need the proof on the event $\m E^c = \{\big|\!\big|\wht \Delta_N\big|\!\big|_E > \delta_N\}$.

\noindent \underline{\bf Step 1: decoupling.} The first technical difficulty is to decouple the interaction effect between particles that would cause the dependence of particle trajectories. Let ${\mb P}^N$ be the probability measure on $(\m C, \{\m F_t\}_{t=0}^T)$ induced by solution of the true data generating mechanism~\eqref{eqn: interact_system}. Motivated by the change of measure argument in \cite{lacker2018strong} and \cite{della2021nonparametric}, we construct a new measure $\barP$ on $(\m C, \{\m F\}_{t=0}^T)$, under which $(X^i_t)_{t=0}^T$ and $(X^j_t)_{t=0}^T$ are independent for all $1\leq i\neq j\leq N$. This is possible thanks to Girsanov's Theorem. Specifically, define a process $\{Z_t\}_{t=0}^T$ as
\begin{align*}
    Z_t = \exp\bigg\{\sum_{i=1}^N\bigg[\int_0^t\!\big\langle b^\ast(\mu_s, X_s^i) - b^\ast(\mu_s^N, X_s^i),\dd W_s^i\big\rangle\! - \!\frac{1}{2}\int_0^t\!\|b^\ast(\mu_s, X_s^i) - b^\ast(\mu_s^N, X_s^i)\|^2\dd s\bigg]\bigg\},
\end{align*}
and let $\,\dd \barP = Z_T\,\dd\mb P^N$. By Girsanov's Theorem,
\begin{align}\label{eqn:transformed_BM}
    \overline{W}_t^i := W_t^i - \int_0^t \Big[b^\ast(\mu_s, X_s^i) - b^\ast(\mu_s^N, X_s^i)\Big]\,\dd s,\quad 0\leq t\leq T
\end{align}
are i.i.d.~Brownian motions on $\mb T^d$ under $\barP$. With the transformation~\eqref{eqn:transformed_BM}, the original IPS (\ref{eqn: interact_system}) turns into a system of \emph{independent} SDEs given by
\begin{align}\label{eqn: mean_field_SDE}
\dd X_t^i = b^\ast(\mu_t, X_t^i)\,\dd t + \,\dd\overline{W}_t^i
\end{align}
with i.i.d.~initialization $\mathcal{L}(X_0^1,\dots,X_0^N) = \otimes_{i=1}^N \mu_0$. Our subsequent strategy for analyzing the log-likelihood ratio is to control the probability of some "bad events" under $\barP$, and then to use the following Lemma~\ref{lem:converting_lemma} to convert it back to the probability under $\mb P^N$. The proof of Lemma~\ref{lem:converting_lemma} can be found in Theorem 18 in \cite{della2021nonparametric}.
\begin{lemma}[Change of measure equivalence]\label{lem:converting_lemma}
There are positive constants $\kappa_1$ and $\kappa_2$ such that for any $\m F_T$-measurable event $\m B$,
\begin{align*}
    \mb P^N(\m B) \leq \kappa_1\barP(\m B)^{\kappa_2}.
\end{align*}
\end{lemma}

\noindent \underline{\bf Step 2: basic inequality.} 
By definition of $\wht b_N$, we have $L_T(\wht b_N) \geq L_T(b^\ast)$ and thus
\begin{align*}
  & \sum_{i=1}^N\int_0^T\big\langle (\wht b_N - b^\ast)( \mu_t^N, X_t^i), \, b^\ast(\mu_t, X_t^i)\,\dd t + \dd\overline{W}_t^i\big\rangle \\
   &\qquad\qquad\qquad\qquad-\frac{1}{2}\sum_{i=1}^N\int_0^T \Big\{\|\wht b_N(\mu_t^N, X_t^i)\|^2 - \|b^\ast(\mu_t^N, X_t^i)\|^2 \Big\}\,\dd t \geq 0.
\end{align*}
Denote $\wht\Delta_N = \wht b_N - b^\ast$, we can derive the basic inequality
\begin{align}\label{eqn: basic_ineq}
\begin{aligned}
    \frac{1}{N}\sum_{i=1}^N\int_0^T\big\langle\wht \Delta_N( \mu_t^N, X_t^i),\,\dd \overline{W}_t^i\big\rangle 
    &\geq \frac{1}{2N}\sum_{i=1}^N\int_0^T\|\wht b_N(\mu_t^N, X_t^i) - b^\ast(\mu_t, X_t^i)\|^2\,\dd t\\ 
    &\qquad- \frac{1}{2N}\sum_{i=1}^N\int_0^T\|b^\ast(\mu_t^N, X_t^i) - b^\ast(\mu_t, X_t^i)\|^2\,\dd t.
\end{aligned}
\end{align}
Furthermore, by noticing that
\begin{align*}
    &\quad\,\,\, \wht b_N(\mu_t^N, X_t^i) - b^\ast(\mu_t, X_t^i)\\
    &= \Big[\wht\Delta_N(\mu_t^N, X_t^i) - \wht\Delta_N(\mu_t, X_t^i)\Big] + \Big[b^\ast(\mu_t^N, X_t^i) - b^\ast(\mu_t, X_t^i)\Big] + \Big[\wht b_N(\mu_t, X_t^i) - b^\ast(\mu_t, X_t^i)\Big],
\end{align*}
we can obtain by using the Cauchy-Schwartz inequality and basic inequality~\eqref{eqn: basic_ineq} that
\begin{align}\label{eqn: new_basic_ineq}
\begin{aligned}
    \frac{1}{N}\sum_{i=1}^N\int_0^T\big\langle\wht\Delta_N(\mu_t, X_t^i),\,\dd \overline{W}_t^i\big\rangle
    &\geq \frac{1}{6N}\sum_{i=1}^N\int_0^T\Big\|\wht\Delta_N(\mu_t, X_t^i)\Big\|^2\,\dd t\\
    &\quad - \frac{1}{N}\sum_{i=1}^N\int_0^T\Big\|b^\ast( \mu_t^N, X_t^i) - b^\ast(\mu_t, X_t^i)\Big\|^2\,\dd t\\
    &\quad - \frac{1}{2N}\sum_{i=1}^N\int_0^T\Big\|\wht\Delta_N(\mu_t^N, X_t^i) - \wht\Delta_N(\mu_t, X_t^i)\Big\|^2\,\dd t\\
    &\quad - \frac{1}{N}\sum_{i=1}^N\int_0^T\big\langle\wht\Delta_N(\mu_t^N, X_t^i) - \wht\Delta_N(\mu_t, X_t^i),\,\dd\overline{W}_t^i\big\rangle.
\end{aligned}
\end{align}
We will call the last three terms on the right hand side of the above display as \emph{decoupling errors} since they characterize the degree of dependence among the particle trajectories due to the presence of empirical law $\mu_t^N$ in the drift term $b^\ast$ of SDE~\eqref{eqn: interact_system}.
To bound the left-hand side of~\eqref{eqn: new_basic_ineq}, we need the following Lemma~\ref{lem: order_of_ito}.
\begin{lemma}[Localization]\label{lem: order_of_ito}
For the star-shaped set $\m H^\ast$, define
\begin{align*}
    \m A(u) = \bigg\{\exists\, g\in\m H^\ast, \norm{g}_E \geq u: \bigg|\frac{1}{N}\sum_{i=1}^N\int_0^T\big\langle g(\mu_t, X_t^i),\,\dd \overline{W}_t^i\big\rangle\bigg| \geq 4u\norm{g}_E\bigg\},
\end{align*}
and critical radius $\delta_N$ as the smallest number $u > 0$ such that
\begin{align*}
    \mb E_{\barP}\sup_{g\in\m H^\ast_{u}}\bigg|\frac{1}{N}\sum_{i=1}^N\int_0^T\big\langle g(\mu_t, X_t^i),\,\dd\overline{W}_t^i\big\rangle\bigg| \leq 2u^2.
\end{align*}
Recall $H(\varepsilon, \m H^\ast_u)$ is the cardinality of the smallest set $S\subset \m H^\ast_u$, such that $\forall\,g\in\m H_u^\ast$, there exists $f\in S$ satisfying $\norm{f-g}_E \leq \varepsilon u$ and $\norm{f-g}_\infty \leq \varepsilon B$. Then, for every $u\geq \delta_N$ we have
\begin{align*}
    \barP\big(\m A(u)\big) \leq \exp\bigg\{-\frac{N\delta_N^2}{3}\bigg\} + 3\exp\bigg\{-\frac{Nu\delta_N}{C\log NJ_1(1)}\bigg\}.
\end{align*}
\end{lemma}
Thus, on the event $\m A(\delta_N)^c$ where $\delta_N$ is the critical radius of our nonparametric estimation problem, we have
\begin{equation}\label{eqn:final_basic_inequality}
    \frac{1}{6N}\sum_{i=1}^N\int_0^T\Big\|\wht\Delta_N(\mu_t, X_t^i)\Big\|^2\,\dd t \leq 4 \delta_N \|\hat\Delta_N\|_E + T_1 + 0.5 T_2 + T_3,
\end{equation}
where the three error terms
\begin{eqnarray*}
T_1 &=& \frac{1}{N}\sum_{i=1}^N\int_0^T\Big\|b^\ast( \mu_t^N, X_t^i) - b^\ast(\mu_t, X_t^i)\Big\|^2\,\dd t, \\
T_2 &=& \sup_{g \in \m H^*} \frac{1}{N}\sum_{i=1}^N\int_0^T\Big\|g( \mu_t^N, X_t^i) - g(\mu_t, X_t^i)\Big\|^2\,\dd t, \\
T_3 &=& \frac{1}{N}\sum_{i=1}^N\int_0^T\big\langle\wht\Delta_N(\mu_t^N, X_t^i) - \wht\Delta_N(\mu_t, X_t^i),\,\dd\overline{W}_t^i\big\rangle.
\end{eqnarray*}
are expected to decay to zero as $N\to\infty$ since $\mu_t^N$ is expected to be close to $\mu_t$. However, a rigorous analysis of these three error terms requires substantial efforts due to their complicated structures and the need of a uniform control over $\m H^*$, which will be sketched below.

\noindent \underline{\bf Step 3: bound $T_1$.} To bound $T_1$, we use the following lemma. Recall that $|\!|f|\!|_{\textrm{Lip}}$ denotes the Lipschitz constant of $f$.

\begin{lemma}[Decoupling error bound]\label{lem: decoupling_err}
For any $u > 0$, $N\geq 2$ and $g\in\m H^\ast$, we have
\begin{align*}
    \barP\bigg(\frac{1}{N}\sum_{i=1}^N\int_0^T\big|\!\big|g(\mu_t, X_t^i) - g(\mu_t^N, X_t^i)\big|\!\big|^2\,\dd t > u^2\bigg) \leq 2e^{-\frac{Nu^2}{4TK_1|\!|\tilde{g}|\!|^2_{\textrm{Lip}}}}.
\end{align*}
\end{lemma}
Applying Lemma~\ref{lem: decoupling_err} with $g = b^*$ and $u^2 = \log(N)/N$, we can conclude that $T_1 \leq \frac{\log N}{N}$ holds with probability at least $1-2\exp\big(-{\log{n} \over 4 T K_1 \|\tilde{b^\ast}\|^2_{\textrm{Lip}}}\big)$.
Lemma~\ref{lem: decoupling_err} tells us that the decoupling error is sub-Gaussian with parameter of order $O(N^{-1})$, although the summands are not independent. In fact, we can decompose
\begin{align*}
    g(\mu_t^N, X_t^i) - g(\mu_t, X_t^i) 
    &= \frac{1}{N}\bigg[\tilde{g}(X_t^i, X_t^i) - \int_{\mb T^d}\tilde{g}(X_t^i, y)\,\dd\mu_t(y)\bigg]\\
    &\qquad\quad + \frac{1}{N}\sum_{j=1, j\neq i}^N\bigg[\tilde{g}(X_t^i, X_t^j) - \int_{\mb T^d}\tilde{g}(X_t^i, y)\,\dd\mu_t(y)\bigg].
\end{align*}
Notice that the second term above is a summation of $(N-1)$ i.i.d.~centered random variables under $\barP$ conditioning on $X_t^i$. This is where the sub-Gaussianity comes from. Intuitively, the second decoupling error $T_2$ (or third line of~\eqref{eqn: new_basic_ineq}) should have the same order $O(N^{-1})$ by some discretization or chaining arguments. This will be done in Lemma \ref{lem: maximal_ineq_decoupling_err}.

\noindent \underline{\bf Step 4: bound $T_2$.} To bound $T_2$, we use the following lemma.

\begin{lemma}[Uniform laws of dependent variables]\label{lem: maximal_ineq_decoupling_err}
We have
\begin{align*}
    \barP\bigg(\sup_{g\in\m H^\ast}\frac{1}{N}\sum_{i=1}^N\int_0^T\norm{g(\mu_t, X_t^i) - g(\mu_t^N, X_t^i)}^2\,\dd t > u\bigg) \leq 2e^{-\frac{Nu}{CLK_1J_2(L)}}
\end{align*}
for some constant $C > 0$. 
\end{lemma}

\begin{remark}
In order for $J_2(L)$ to be finite, we need functions in $\m H^\ast$ to have higher-order smoothness than just being Lipschitz continuous, so that the covering number with respect to $\norm{\tilde{\cdot}}_\Lip$ is finite.
\end{remark}

\noindent \underline{\bf Step 5: bound $T_3$.} The last decoupling error $T_3$ in~\eqref{eqn:final_basic_inequality} (or last line of~\eqref{eqn: new_basic_ineq}) is much more involved to control. This is because the Ito integral is expected to only have the order of the square root of its quadratic variation. Notice that for each $g(\nu, x) = \int_{\mb T^d}\tilde{g}(x, y)\,\dd\nu(y)$, we have
\begin{align}\label{eqn: decompose_Ito_term}
\begin{aligned}
    &\quad\,\frac{1}{N}\sum_{i=1}^N\int_0^T\big\langle g(\mu_t^N, X_t^i) - g(\mu_t, X_t^i),\,\dd\overline{W}_t^i\big\rangle\\
    &= \frac{1}{N^2}\sum_{i=1}^N\int_0^T\big\langle\xi_{i,i}^g(t),\,\dd\overline{W}_t^i\big\rangle + \frac{1}{N^2}\sum_{1\leq i\neq j\leq N} \int_0^T\big\langle\xi_{i,j}^g(t),\,\dd\overline{W}_t^i\big\rangle =: V_N(g) + U_N(g),
\end{aligned}
\end{align}
where we used the shorthand $\xi_{i,j}^g(t) = \tilde{g}(X_t^i, X_t^j) - \int_{\mb T^d}\tilde{g}(X_t^i, y)\,\dd\mu_t(y)$ for notation simplicity. The first term $V_N(g)$ in (\ref{eqn: decompose_Ito_term}) is a sum of i.i.d.~random variable, which is expected to have order $O(N^{-\frac{3}{2}})$. The second term $U_N(g)$ can be viewed as a U-statistics (after proper symmetrization), with kernel function
$g^\dagger(\overline{W}^i, \overline{W}^j) :=  \int_0^T\big\langle\xi_{i,j}^g(t),\,\dd\overline{W}_t^i\big\rangle$.
It is easy to verify that $\mb E_{\barP}\big[g^\dagger(\overline{W}^i, \overline{W}^j)\,|\,\overline{W}^i\big] = \mb E_{\barP}\big[g^\dagger(\overline{W}^i, \overline{W}^j)\,|\,\overline{W}^j\big] = 0.$ So $g^\dagger$ is degenerate (Definition 3.5.1 in \cite{de1999decoupling}), indicating that the U-statistics should have order around $O(N^{-1})$ asymptotically (see Chapter 3 of \cite{lee1990u}). Above discussions can be rigorously stated as Lemma \ref{lem: diagonal_elements} and Lemma \ref{lem: U_statistics} under a non-asymptotic setting. Then the last term $T_3$ in~\eqref{eqn:final_basic_inequality} can be bounded by the following lemma, which gives an optimal upper bound $O(N^{-1})$ of the decoupling error up to some log factor.
\begin{lemma}[Uniform laws of dependent It\^o integrals]\label{lem: decoupling_err_Ito}
We have
\begin{align*}
    \sup_{g\in\m H^\ast}\frac{1}{N}\sum_{i=1}^N\int_0^T\big\langle g(\mu_t^N, X_t^i) - g(\mu_t, X_t^i),\,\dd\overline{W}_t^i\big\rangle \leq \frac{2\log N}{N}
\end{align*}
with probability at least
\begin{align*}
    1 - 2\exp\bigg\{-\frac{N(\log N)^2}{CJ_3^2(B)}\bigg\} - 2\exp\bigg\{-\Big(\frac{\log N}{CJ_4(B)\log\log N}\Big)^{\frac{2}{3}}\bigg\}.
\end{align*}
\end{lemma}

\noindent \underline{\bf Step 6: conclude.} To finish the bound for the estimation error under the $\|\cdot\|_E$ norm, we need the following norm equivalence between $\|\cdot\|_E$ and its empirical version $\|\cdot\|_X$.
\begin{lemma}[Equivalence of norms]\label{lem: equivalence_norm}
There is a constant $C > 0$ such that
\begin{align*}
    \barP\bigg(\sup_{g\in\m H^\ast}\Big|\norm{g}_X^2 - \norm{g}_E^2\Big| > \frac{1}{2}\norm{g}_E^2 + \frac{u^2}{2}\bigg) < \exp\bigg\{-\frac{Nu^2}{36C^2B^2T}\bigg\}
\end{align*}
for all $u \geq r_N$. Here $r_N > 0$ is a constant satisfying $J_5(r_N) \leq \frac{\sqrt{N}r_N^2}{6CB\sqrt{T}}.$
\end{lemma}

Return to the proof. Since $\big|\!\big|\wht \Delta_N\big|\!\big|_E > \delta_N$, we can obtain by combining all pieces with~\eqref{eqn:final_basic_inequality} that
\begin{align*}
    4\delta_N\big|\!\big|\wht\Delta_N\big|\!\big|_E
    \stackrel{\ri}{\geq}\frac{1}{6}\big|\!\big|\wht\Delta_N\big|\!\big|_X^2 - \frac{\log N}{N} - \frac{\log N}{2N} - \frac{2\log N}{N} \stackrel{\rii}{\geq} \frac{1}{12}\big|\!\big|\wht \Delta_N\big|\!\big|_E^2 - \frac{r_N^2}{2} - \frac{7\log N}{N}
\end{align*}
with at least probability given in~\eqref{eqn: tail_prob}. Here, (i) is by Lemma \ref{lem: order_of_ito} (taking $u=\delta_N$),  Lemma \ref{lem: decoupling_err} (taking $u^2 = \frac{\log N}{N}$), Lemma \ref{lem: maximal_ineq_decoupling_err} (taking $u = \frac{\log N}{N}$), and Lemma \ref{lem: decoupling_err_Ito}, and (ii) is by Lemma \ref{lem: equivalence_norm} (taking $u = r_N$). This implies the desired bound~\eqref{eqn: converge_of_b}.

\acks{Xiaohui Chen was partially supported by NSF CAREER Award DMS-1752614. Yun Yang was partially supported by NSF DMS-1907316.}

\bibliography{yourbibfile}

\newpage
\appendix
\begin{center}
{\bf\Large Appendix}
\end{center}

\section{Proof of Main Results}

\begin{proof}[Proof of Lemma \ref{lem: order_of_ito}]
Let
\begin{align*}
    \m Z_N(u) = \sup_{g\in\m H^\ast_u} \bigg|\frac{1}{N}\sum_{i=1}^N\int_0^T\big\langle g(\mu_t, X_t^i), \,\dd\overline{W}_t^i\big\rangle\bigg|    
\end{align*}
First, we shall show that
\begin{align}\label{eqn: reformulate_to_GaussComplex}
    \m A(u) \subset \big\{\m Z_N(u) \geq C_3u^2\big\}.
\end{align}
In fact, assume there is $g\in\m H^\ast$, such that $\norm{g}_E\geq u$ and
\begin{align}\label{eqn: gaussian_complexity1}
    \bigg|\frac{1}{N}\sum_{i=1}^N\int_0^T\big\langle g(\mu_t, X_t^i),\,\dd\overline{W}_t^i\big\rangle\bigg| \geq C_3u\norm{g}_E.
\end{align}
Let $h = \frac{u}{\norm{g}_E}g$. Since $\norm{g}_E\geq u$ and $\m H^\ast$ is star-shaped, we know $h\in\m H^\ast$ and $\norm{h}_E = u$. Then (\ref{eqn: gaussian_complexity1}) can be reformulated as
\begin{align*}
    \bigg|\frac{1}{N}\sum_{i=1}^N\int_0^T\big\langle h(\mu_t, X_t^i),\,\dd\overline{W}_t^i\big\rangle\bigg| \geq C_3u^2.
\end{align*}
Now we have finished (\ref{eqn: reformulate_to_GaussComplex}). By a deviation inequality (Theorem 4 in \cite{adamczak2008tail}),
\begin{align*}
    &\quad\,\barP\Big(\m Z_N(u) \geq 1.5\mb E_{\barP}\m Z_N(u) + t\Big)\\ 
    &\leq \exp\bigg\{-\frac{t^2}{3\sigma^2}\bigg\} + 3\exp\bigg\{-\frac{t}{C\big|\!\big|\max_i\sup_{g\in\m H^\ast_u}\frac{1}{N}\int_0^T\langle g(\mu_t, X_t^i),\dd\overline{W}_t^i\rangle\big|\!\big|}_{\psi_1}\bigg\},
\end{align*}
with
\begin{align*}
    \sigma^2 &= \sup_{g\in\m H^\ast_u}\frac{1}{N^2}\sum_{i=1}^N\mb E_{\barP}\bigg(\int_0^T\big\langle g(\mu_t, X_t^i),\,\dd\overline{W}_t^i\big\rangle\bigg)^2 \leq \sup_{g\in\m H^\ast_u} \frac{1}{N}\norm{g}_E^2 \leq \frac{u^2}{N}.
\end{align*}
Then take $t = u\delta_N$ in the above deviation inequality, and we have
\begin{align}
\begin{aligned}
    &\quad\,\barP\Big(\m Z_N(u) \geq 1.5\mb E_{\barP}\m Z_N(u) + u\delta_N\Big)\\
    &\leq \exp\bigg\{-\frac{N\delta_N^2}{3}\bigg\} + 3\exp\bigg\{-\frac{Nu\delta_N}{C\big|\!\big|\max_i\sup_{g\in\m H^\ast_u}\int_0^T\langle g(\mu_t, X_t^i),\dd\overline{W}_t^i\rangle\big|\!\big|}_{\psi_1}\bigg\}.
\end{aligned}
\end{align}
Notice that we have
\begin{align*}
    &\quad\,\norm{\max_{1\leq i\leq N}\sup_{g\in\m H^\ast_u}\int_0^T\big\langle g(\mu_t, X_t^i),\,\dd\overline{W}_t^i\big\rangle}_{\psi_1}\\ 
    &\stackrel{\ri}{\lesssim}\log N\max_{1\leq i\leq N}\norm{\sup_{g\in\m H^\ast_u}\int_0^T\big\langle g(\mu_t, X_t^i),\,\dd\overline{W}_t^i\big\rangle}_{\psi_1}\\
    &\leq \log N\bigg|\!\bigg| \sup_{g\in\m H^\ast_u}\int_0^T\big\langle g(\mu_t, X_t^i), \,\dd\overline{W}_t^i\big\rangle\bigg|\!\bigg|_{\psi_1}\\
    &\stackrel{\rii}{\lesssim} \log N\bigg[\sqrt{T}B\int_0^\frac{1}{2}\log\big(1 + H(\varepsilon, \m H^\ast)\big)\,\dd \varepsilon + \int_0^\frac{1}{2}\sqrt{\log\big(1 + H(\varepsilon,\m H^\ast_u)\big)}\,\dd \varepsilon\bigg]\\
    &= \log N \cdot J_1(1).
\end{align*}
Here (i) is by Lemma 8.2 in \cite{kosorok2008introduction}, and (ii) is by taking $N=1$ in Lemma \ref{lem: psi1_chaining}. Since $u\geq \delta_N$, by Lemma~\ref{lem: monotonicity} we know
\begin{align*}
    \frac{\mb E_{\barP}\m Z_N(u)}{u} \leq \frac{\mb E_{\barP}\m Z_N(\delta_N)}{\delta_N} \leq 2\delta_N,
\end{align*}
i.e. $\mb E_{\barP}\m Z_N(u) \leq 2u\delta_N$. Therefore,
\begin{align*}
    &\quad\,\barP\big(\m A(u)\big) \leq \barP\big(\m Z_N(u) \geq 4u^2\big) \leq \barP\big(\m Z_N(u) \geq 4u\delta_N\big)\\
    &\leq \exp\bigg\{-\frac{N\delta_N^2}{3}\bigg\} + 3\exp\bigg\{-\frac{Nu\delta_N}{C\log NJ_1(u)}\bigg\}.
\end{align*}
We finish the proof.
\end{proof}

\begin{proof}[Proof of Lemma \ref{lem: decoupling_err}]
Note that
\begin{align*}
    &\quad\, \mb E_{\barP}\exp\bigg\{\frac{1}{N\lambda^2}\sum_{i=1}^N\int_0^T\Big\|g(\mu_t, X_t^i) - g(\mu_t^N, X_t^i)\Big\|^2\,\dd t\bigg\}\\
    &\stackrel{\ri}{\leq} \mb E_{\barP}\int_0^T\exp\bigg\{\frac{T}{N\lambda^2}\sum_{i=1}^N\Big\|g(\mu_t, X_t^i) - g(\mu_t^N, X_t^i)\Big\|^2\bigg\}\frac{\dd t}{T}\\
    &\stackrel{\rii}{\leq} \int_0^T\mb E_{\barP}\exp\bigg\{\frac{T}{\lambda^2}\Big\|g(\mu_t, X_t^N) - g(\mu_t^N, X_t^N)\Big\|^2\bigg\}\frac{\dd t}{T}\\
    &\leq \sup_{0\leq t\leq T}\mb E_{\barP}\exp\bigg\{\frac{T}{\lambda^2}\Big\|g(\mu_t, X_t^N) - g(\mu_t^N, X_t^N)\Big\|^2\bigg\}.
\end{align*}
Here (i) is by Jensen's inequality, and (ii) is by H\"{o}lder's inequality for multivariables, which implies $\mb E_{\barP}Y_1\cdots Y_N \leq \mb E_{\barP}Y_N^N$, since $Y_i = \exp\big\{\frac{T}{N\lambda^2}\|g(\mu_t, X_t^i) - g(\mu_t^N, X_t^i)\|^2\big\}$ are identically distributed (no need to be independent) under $\barP$ for $i=1,2,\cdots, N$. By Taylor's expansion,
\begin{align*}
    &\quad\,\mb E_{\barP}\exp\bigg\{\frac{T}{\lambda^2}\Big\|g(\mu_t, X_t^N) - g(\mu_t^N, X_t^N)\Big\|^2\bigg\}\\
    &\leq 1 + \sum_{p\geq 1}\frac{T^p}{p!\lambda^{2p}}\mb E_{\barP}\Big\|g(\mu_t, X_t^N) - g(\mu_t^N, X_t^N)\Big\|^{2p}\\
    &\stackrel{\ri}{\leq} 1 + \sum_{p\geq 1} \frac{T^p}{p!\lambda^{2p}}\cdot \frac{p!K_1^p}{(N-1)^p}\|\tilde{g}\|_\Lip^{2p}\\
    &= \frac{1}{1 - \frac{TK_1\|\tilde{g}\|_\Lip^2}{(N-1)\lambda^2}}
\end{align*}
for all $\lambda^2 > TK_1\|\tilde{g}\|_\Lip^2/(N-1)$. Here (i) is by Lemma \ref{lem: bd_2pmoment}. Therefore, we get
\begin{align*}
    \mb E_{\barP}\exp\bigg\{\frac{1}{N\lambda^2}\sum_{i=1}^N\int_0^T\Big\|g(\mu_t, X_t^i) - g(\mu_t^N, X_t^i)\Big\|^2\,\dd t\bigg\} \leq \frac{1}{1 - \frac{TK_1\|\tilde{g}\|_\Lip^2}{(N-1)\lambda^2}}.
\end{align*}
This implies
\begin{align}\label{eqn: orlicz2_Zprocess}
    \norm{\bigg(\frac{1}{N}\sum_{i=1}^N\int_0^T\Big\|g(\mu_t, X_t^i) - g(\mu_t^N, X_t^i)\Big\|^2\,\dd t\bigg)^{\frac{1}{2}}}_{\psi_2} \leq 2\|\tilde{g}\|_\Lip\sqrt{\frac{TK_1}{N}}.
\end{align}
Therefore, 
\begin{align*}
    \barP\bigg(\frac{1}{N}\sum_{i=1}^N\int_0^T\big|\!\big|g(\mu_t, X_t^i) - g(\mu_t^N, X_t^i)\big|\!\big|^2\,\dd t > u^2\bigg) \leq 2e^{-\frac{Nu^2}{4TK_1|\!|\tilde{g}|\!|^2_{\textrm{Lip}}}}
\end{align*}
\end{proof}

\begin{proof}[Proof of Lemma \ref{lem: maximal_ineq_decoupling_err}]
For simplicity, let
\begin{align*}
    Z_g := \frac{1}{N}\sum_{i=1}^N\int_0^T\norm{g(\mu_t, X_t^i) - g(\mu_t^N, X_t^i)}^2\,\dd t.
\end{align*}
Notice that
\begin{align*}
    &\quad\,\big|Z_f - Z_g\big|\\ 
    &= \bigg|\frac{1}{N}\sum_{i=1}^N\int_0^T\big\langle(f-g)(\mu_t, X_t^i)-(f-g)(\mu_t^N, X_t^i), (f+g)(\mu_t, X_t^i)-(f+g)(\mu_t, X_t^i)\big\rangle\,\dd t\bigg|\\
    &\leq \sqrt{Z_{f-g}}\cdot\sqrt{Z_{f+g}}.
\end{align*}
Then, we have
\begin{align*}
    \norm{Z_f - Z_g}_{\psi_1} &\stackrel{\ri}{\leq} \norm{\sqrt{Z_{f-g}}}_{\psi_2}\cdot\norm{\sqrt{Z_{f+g}}}_{\psi_2}\\
    &\stackrel{\rii}{\lesssim} \norm{\tilde{f} - \tilde{g}}_{\textrm{Lip}}\sqrt{\frac{
    K_1}{N}}\cdot \norm{\tilde{f} + \tilde{g}}_{\textrm{Lip}}\sqrt{\frac{K_1}{N}}\\
    &\lesssim \frac{LK_1}{N}\norm{\tilde{f} - \tilde{g}}_{\textrm{Lip}}.
\end{align*}
Here, (i) is by Lemma \ref{lem: orlicznorm_prod} and (ii) is by (\ref{eqn: orlicz2_Zprocess}) in the proof of Lemma \ref{lem: decoupling_err}. Then, by the standard chaining argument, we know
\begin{align*}
    \norm{\sup_{g\in\m H^\ast} Z_g}_{\psi_1} \lesssim \frac{LK_1}{N}\int_0^{\frac{L}{2}}\log\big(1 + N(\varepsilon, \m H^\ast, \norm{\tilde{\cdot}}_{\textrm{Lip}})\big)\,\dd\varepsilon = \frac{LK_1}{N}J_2(L).
\end{align*}
Therefore, we have
\begin{align*}
    \barP\Big(\sup_{g\in\m H^\ast}Z_g > u\Big) \leq 2e^{-\frac{Nu}{CLK_1J_2(L)}}
\end{align*}
for some constant $C > 0$.
\end{proof}

\begin{proof}[Proof of Lemma \ref{lem: decoupling_err_Ito}]
Recall that
\begin{align*}
    &\quad\,\sup_{g\in\m H^\ast}\frac{1}{N}\sum_{i=1}^N\int_0^T\big\langle g(\mu_t^N, X_t^i) - g(\mu_t, X_t^i),\,\dd \overline{W}_t^i\big\rangle\\
    &\leq \sup_{g\in\m H^\ast}\frac{1}{N^2}\sum_{i=1}^N\int_0^T\big\langle\xi_{i,i}^g(t),\,\dd\overline{W}_t^i\big\rangle + \sup_{g\in\m H^\ast}\frac{1}{N^2}\sum_{1\leq i\neq j\leq N}\int_0^T\big\langle\xi_{i,j}^g(t)\,\dd\overline{W}_t^i\big\rangle\\
    &=\sup_{g\in\m H^\ast} V_N(g) + \sup_{g\in\m H^\ast} U_N(g).
\end{align*}
We will start from $\sup_{g\in\m H^\ast}V_N(g)$. By Lemma \ref{lem: diagonal_elements} and standard chaining argument, 
\begin{align*}
    \norm{\sup_{g\in\m H^\ast}V_N(g)}_{\psi_2} \lesssim \frac{1}{N\sqrt{N}}\int_0^\frac{B}{2}\sqrt{\log\big(1 + N(u, \m H^\ast, \norm{\cdot}_\infty)\big)}\,\dd u = \frac{J_3(B)}{N\sqrt{N}}.
\end{align*}
This implies that
\begin{align}\label{eqn: high_prob_bound_Vn}
    \barP\bigg(\sup_{g\in\m H^\ast}V_N(g) > u\bigg) \leq 2e^{-\frac{N^3u^2}{CJ_3^2(B)}}
\end{align}
for some constant $C$. Next, we will bound $\sup_{g\in\m H^\ast}U_N(g)$. By Lemma \ref{lem: U_statistics},
\begin{align*}
    \norm{U_N(g)}_{\psi_{2/3}} \lesssim \frac{B\log\log N}{N}.
\end{align*}
Notice that we cannot directly use chaining method with respect to $\psi_{\frac{2}{3}}$-norm, since $\psi_\beta$ is not convex when $0 < \beta<1$. Luckily, by Lemma C.2 in \cite{chen2019randomized}, there is a convex function $\tilde{\psi}_{\frac{2}{3}}$, such that $\norm{\cdot}_{\psi_{2/3}}$ and $\norm{\cdot}_{\tilde{\psi}_{2/3}}$ are equivalent. In fact, a possible construction is 
\begin{align*}
    \tilde{\psi}_{\frac{2}{3}}(x) = 
    \begin{cases}
        \frac{\psi_{2/3}(u)}{u}x & x\leq u\\
        \psi_{\frac{2}{3}}(x) & x > u,
    \end{cases}
\end{align*}
where $u$ is the unique solution to $3(1 - e^{-u}) = 2u$. Thus
\begin{align*}
    \norm{U_N(g)}_{\tilde{\psi}_{2/3}} \lesssim \frac{B\log\log N}{N}.
\end{align*}
We can then verify that $\tilde{\psi}_{2/3}^{-1} \lesssim \psi_{2/3}^{-1} = [\log(1+x)]^{3/2}$. Again, by a chaining argument, 
\begin{align*}
    \norm{\sup_{g\in\m H^\ast}U_N(g)}_{\psi_{2/3}} &\lesssim \norm{\sup_{g\in\m H^\ast}U_N(g)}_{\tilde{\psi}_{2/3}}\\ 
    &\lesssim \frac{\log\log N}{N}\int_0^{\frac{B}{2}} \tilde{\psi}_{\frac{2}{3}}^{-1}\big(N(u,\m H^\ast, \norm{\cdot}_\infty)\big)\,\dd u\\
    &\lesssim \frac{\log\log N}{N}\int_0^\frac{B}{2} \big[\log(1 + N(u,\m H^\ast, \norm{\cdot}_\infty))\big]^{\frac{3}{2}}\,\dd u\\
    &= \frac{\log\log N}{N}J_4(B).
\end{align*}
It implies that
\begin{align}\label{eqn: high_prob_bound_Un}
    \barP\bigg(\sup_{g\in\m H^\ast}V_N(g) > u\bigg) < 2\exp\bigg\{-\Big(\frac{Nu}{CJ_4(B)\log\log N}\Big)^{\frac{2}{3}}\bigg\}
\end{align}
for some constant $C>0$. Taking $u = \frac{\log N}{N}$ in (\ref{eqn: high_prob_bound_Vn}) and (\ref{eqn: high_prob_bound_Un}),  we have
\begin{align*}
    \sup_{g\in\m H^\ast} V_N(g) + \sup_{g\in\m H^\ast} U_N(g) \leq \frac{2\log N}{N}
\end{align*}
with probability at least
\begin{align*}
    1 - 2\exp\bigg\{-\frac{N(\log N)^2}{CJ_3^2(B)}\bigg\} - 2\exp\bigg\{-\Big(\frac{\log N}{CJ_4(B)\log\log N}\Big)^{\frac{2}{3}}\bigg\}.
\end{align*}
\end{proof}

\begin{proof}[Proof of Lemma \ref{lem: equivalence_norm}]
Recall that
\begin{align*}
    \norm{g}_X^2 = \frac{1}{N}\sum_{i=1}^N\int_0^T\norm{g(\mu_t, X_t^i)}^2\,\dd t, \qquad \norm{g}_E^2 = \mb E_{\barP}\norm{g}_X^2.
\end{align*}
First, we will show that
\begin{align}\label{eqn: prob_diff_xnorm_enorm}
    \barP\bigg(\exists\, g\in\m H^\ast: \Big|\norm{g}_X^2 - \norm{g}_E^2\Big| > \frac{1}{2}\norm{g}_E^2 + \frac{u^2}{2}\bigg) \leq \barP\bigg(\sup_{g\in\m H^\ast_u}\Big|\norm{g}_X^2 - \norm{g}_E^2\Big| > \frac{u^2}{2}\bigg).
\end{align}
In fact, let $g\in\m H^\ast$ such that
\begin{align*}
    \Big|\norm{g}_X^2 - \norm{g}_E^2\Big| > \frac{1}{2}\norm{g}_E^2 + \frac{u^2}{2}.
\end{align*}
If $\norm{g}_E \leq u$, then it is natural to have $\big|\norm{g}_X^2 - \norm{g}_E^2\big| > u^2/2.$ Otherwise, if $\norm{g}_E < u$, consider $h = \frac{u}{\norm{g}_E}\cdot g$. We know $h\in\m H^\ast$ since $\m H^\ast$ is star-shaped and $\norm{h}_E = u$, and
\begin{align*}
    \Big|\norm{h}_X^2 - \norm{h}_E^2\Big| = \frac{u^2}{\norm{g}_E^2}\Big|\norm{g}_X^2 - \norm{g}_E^2\Big| > \frac{u^2}{2}.
\end{align*}
Thus (\ref{eqn: prob_diff_xnorm_enorm}) is true, and we only need to bound the right-hand side. Since
\begin{align*}
    \bigg|\!\bigg| \frac{1}{N}\int_0^T\norm{g(\mu_t, X_t^i)}^2\,\dd t\bigg|\!\bigg|_\infty \leq \frac{16B^2T}{N},
\end{align*}
and
\begin{align*}
    \sup_{g\in\m H^\ast_u}\sum_{i=1}^N\mb V_{\barP}\bigg(\frac{1}{N}\int_0^T\norm{g(\mu_t, X_t^i)}^2\,\dd t\bigg)
    &\leq \sup_{g\in\m H^\ast_u} \sum_{i=1}^N \mb E_{\barP}\bigg(\frac{1}{N}\int_0^T\norm{g(\mu_t, X_t^i)}^2\,\dd t\bigg)^2\\
    &= \sup_{g\in\m H^\ast_u}N^{-1}\norm{g}_E^2\\
    &\leq \frac{4B^2Tu^2}{N}.
\end{align*}
By Talagrand's inequality~\citep{massart2000constants},
\begin{align}\label{eqn: Talagrand}
\begin{aligned}
    \barP\bigg(\sup_{g\in\m H^\ast_u}\Big|\norm{g}_X^2 - \norm{g}_E^2\Big| &\geq 2\mb E_{\barP}\sup_{g\in\m H^\ast_u}\Big|\norm{g}_X^2 - \norm{g}_E^2\Big| \\
    &\qquad\quad +\frac{4Bu\sqrt{2Tx}}{\sqrt{N}} + \frac{560B^2Tx}{N}\bigg) \leq e^{-x},
\end{aligned}
\end{align}
Now, let us bound the expectation term in the probability. By a standard symmetrization argument, it can be bounded by the Rademacher complexity of $(\m H^\ast_u)^2$, i.e.
\begin{align*}
    \mb E_{\barP} {\sup_{g\in\m H^\ast_u}}\Big|\norm{g}_X^2 - \norm{g}_E^2\Big| \leq 2\mb E_{\barP}\sup_{g\in\m H^\ast_u}\bigg|\frac{1}{N}\sum_{i=1}^N\varepsilon_i\int_0^T\norm{g(\mu_t, X_t^i)}^2\,\dd t\bigg|,
\end{align*}
where $\varepsilon_i$ are i.i.d.~Rademacher random variables. Let
\begin{align*}
    W_g := \frac{1}{\sqrt{N}}\sum_{i=1}^N\varepsilon_i\int_0^T\norm{g(\mu_t, X_t^i)}^2\,\dd t,
\end{align*}
and
\begin{align*}
    \wht R_u^2 := \sup_{g\in\m H^\ast_u}\frac{1}{N}\sum_{i=1}^N\int_0^T\norm{g(\mu_t, X_t^i)}^2\,\dd t = \sup_{g\in\m H^\ast_u}\norm{g}_X^2.
\end{align*}

First, we will show that
\begin{align}\label{eqn: bound_Wg_randomR}
    \mb E_{\barP}\sup_{g\in\m H^\ast_u}|W_g| \leq 32\sqrt{T}B J_5\Big(\sqrt{\mb E_{\barP}\wht R_u^2}\Big).
\end{align}
Here, $u$ indicates that we only consider the covering number of $\m H^\ast_u$ rather than $\m H^\ast$. Notice that for every $f$ and $g\in\m H^\ast_u$,
\begin{align*}
    \mb E_{\barP}^\varepsilon e^{\lambda(W_f - W_g)} &= \mb E_{\barP}^\varepsilon \exp\bigg\{\frac{\lambda}{\sqrt{N}}\sum_{i=1}^N\varepsilon_i\bigg(\int_0^T\norm{f(\mu_t, X_t^i)}^2 - \norm{g(\mu_t, X_t^i)}^2\,\dd t\bigg)\bigg\}\\
    &\leq \exp\bigg\{\frac{\lambda^2}{2N}\sum_{i=1}^N\bigg(\int_0^T\norm{f(\mu_t, X_t^i)}^2 - \norm{g(\mu_t, X_t^i)}^2\,\dd t\bigg)^2\bigg\}\\
    &\leq \exp\bigg\{\frac{\lambda^2}{2N}\sum_{i=1}^N\int_0^T\norm{(f-g)(\mu_t, X_t^i)}^2\,\dd t\cdot \int_0^T\norm{(f+g)(\mu_t, X_t^i)}^2\,\dd t\bigg\}\\
    &\leq \exp\bigg\{8T\norm{f-g}_X^2B^2\lambda^2\bigg\}.
\end{align*}
So $W_f - W_g$ is sub-Gaussian with parameter $16TB^2\norm{f-g}_X^2$. By a standard chaining argument, we have
\begin{align*}
    \mb E_{\barP}^\varepsilon\sup_{g\in\m H^\ast_u}|W_g| &\leq 32\sqrt{T}B\int_0^{\wht R_u/2}\sqrt{\log N(s, \m H^\ast_u, \norm{\cdot}_X)}\,\dd s\\
    &\leq 32\sqrt{T}B\int_0^{\wht R_u/2}\sqrt{\log N(s/\sqrt{T}, \m H^\ast_u, \norm{\cdot}_\infty)}\,\dd s\\
    &= 32\sqrt{T}B J_5(\wht R_u).
\end{align*}
Thus, we have
\begin{align*}
    \mb E_{\barP}\sup_{g\in\m H^\ast_u}|W_g| \leq 32\sqrt{T}B\mb E_{\barP}J_5(\wht R_u) \leq 32\sqrt{T}B J_5\Big(\sqrt{\mb E_{\barP}\wht R_u^2}\Big).
\end{align*}
The last inequality is because $z\mapsto J_5(\sqrt{z})$ is a concave function.

Next, We will estimate $J_5\Big(\sqrt{\mb E_{\barP}\wht R_u^2}\Big)$. Recall that we have shown
\begin{align*}
    \mb E_{\barP}\wht R_u^2 - u^2 &\leq \frac{2}{\sqrt{N}}\mb E_{\barP}\sup_{g\in\m H^\ast_u}|W_g|\\
    &\leq \frac{64\sqrt{T}B}{\sqrt{N}}J_5\Big(\sqrt{\mb E_{\barP}\wht R_u^2}\Big),
\end{align*}
i.e.
\begin{align*}
    \mb E_{\barP}\wht R_u^2 \leq u^2 + \frac{64\sqrt{T}B}{\sqrt{N}}J_5\Big(\sqrt{\mb E_{\barP}\wht R_u^2}\Big).
\end{align*}
Applying Lemma 2.1 in \cite{van2011local} to $J_5(\cdot)$, we get
\begin{align*}
    J_5\Big(\sqrt{\mb E_{\barP}\wht R_u^2}\Big) \lesssim J_5(u)\bigg(1 + J_5(u)\frac{64\sqrt{T}B}{\sqrt{N}u^2}\bigg).
\end{align*}
Taking it back to (\ref{eqn: bound_Wg_randomR}) leads to
\begin{align*}
    \mb E_{\barP}\sup_{g\in\m H^\ast_u}|W_g| &\lesssim \sqrt{T}BJ_5(u)\bigg(1 + J_5(u)\frac{64\sqrt{T}B}{\sqrt{N}u^2}\bigg)\\
    &\lesssim B\sqrt{T}J_5(u) + \frac{B^2TJ_5^2(u)}{\sqrt{N}u^2}
\end{align*}
Combining with (\ref{eqn: Talagrand}), we get
\begin{align*}
    \barP\bigg(C^{-1}\sup_{g\in\m H^\ast_u}\Big|\norm{g}_X^2 - \norm{g}_E^2\Big| \geq \frac{\big[J_5(u) + u\sqrt{x}\big]B\sqrt{T}}{\sqrt{N}} + \frac{\big[xu^2 + J_5^2(u)\big]B^2T}{Nu^2}\bigg) \leq e^{-x}
\end{align*}
for some constant $C > 0$. Take $x = Nu^2/36C^2B^2T$, and notice that
\begin{align*}
    \frac{J_5(u)}{u} \leq \frac{J_r(r_N; r_N)}{r_N} \leq \frac{\sqrt{N}r_N}{6CB\sqrt{T}}.
\end{align*}
Then, we get
\begin{align*}
    \barP\bigg(\sup_{g\in\m H^\ast_u}\Big|\norm{g}_X^2 - \norm{g}_E^2\Big| \geq \frac{ur_N + u^2}{6} + \frac{u^2 + r_N^2}{9}\bigg) \leq \exp\bigg\{-\frac{Nu^2}{36C^2B^2T^2}\bigg\}.
\end{align*}
Again, since $r_N \leq u$ we have
\begin{align*}
    \barP\bigg(\sup_{g\in\m H^\ast_u}\Big|\norm{g}_X^2 - \norm{g}_E^2\Big| > \frac{u^2}{2}\bigg) \leq \exp\bigg\{-\frac{Nu^2}{36C^2B^2T}\bigg\}.
\end{align*}
\end{proof}

\begin{remark}
Usually, there are several ways to bound the Rademacher complexity of $(\m H^\ast_u)^2$. A direct way is using Ledoux--Talagrand contraction inequality, see e.g. the proof of Lemma 14.9 in \cite{wainwright2019high}. The method we use here is similar to the one in \cite{van2014uniform}. 
\end{remark}

\section{Proof of Technical Lemmas}

\begin{lemma}\label{lem: monotonicity}
For any $0 < u_1\leq u_2$, we have
\begin{align*}
    \frac{\mb E_{\barP}\m Z_N(u_2)}{u_2} \leq \frac{\mb E_{\barP}\m Z_N(u_1)}{u_1}.
\end{align*}
\end{lemma}
\begin{proof}
Recall that 
\begin{align*}
    \m Z_N(u) = \sup_{g\in\m H^\ast_u}\bigg|\frac{1}{N}\sum_{i=1}^N\int_0^T\big\langle g(\mu_t, X_t^i), \,\dd\overline{W}_t^i\bigg|
\end{align*}
and $\m H^\ast_u$ is star-shaped for any $u$. Then for any $g\in\m H^\ast_{u_2}$, let $\tilde{g} = \frac{u_1}{u_2}g\in\m H^\ast_{u_1}$ since $0 < \frac{u_1}{u_2}\leq 1.$ So
\begin{align*}
    \frac{u_1}{u_2}\mb E_{\barP}\m Z_N(u_2) &= \mb E_{\barP}\sup_{g\in\m H^\ast_{u_2}}\bigg|\frac{1}{N}\sum_{i=1}^N\int_0^T\Big\langle\frac{u_1}{u_2}g(\mu_t, X_t^i),\,\dd\overline{W}_t^i\Big\rangle\bigg|\\
    &\leq \mb E_{\barP}\sup_{\tilde{g}\in\m H^\ast_{u_1}}\bigg|\frac{1}{N}\sum_{i=1}^N\int_0^T\big\langle\tilde{g}(\mu_t, X_t^i),\,\dd\overline{W}_t^i\big\rangle\bigg|\\
    &= \mb E_{\barP}\m Z_N(u_1).
\end{align*}
\end{proof}

\begin{lemma}[Lemma 22 in \cite{della2021nonparametric}]\label{lem: bd_2pmoment}
For any $g\in\m H^\ast$ we have
\begin{align*}
    \mb E_{\barP}\Big\|g(\mu_t, X_t^N) - g(\mu_t^N, X_t^N)\Big\|^{2p} \leq \frac{p!K_1^p}{(N-1)^p}\|\tilde{g}\|_{\Lip}^{2p}.
\end{align*}
Here $K_1 \lesssim 1 + d^2$ is a constant.
\end{lemma}

\begin{lemma}[$\psi_\alpha$-norm of product]\label{lem: orlicznorm_prod}
For any $\alpha > 0$ and random variables $X$ and $Y$, we have
\begin{align*}
    \norm{XY}_{\psi_\alpha} \leq \norm{X}_{\psi_{2\alpha}}\cdot\norm{Y}_{\psi_{2\alpha}}.
\end{align*}
\end{lemma}
\begin{proof}
By Cauchy--Schwartz's inequality,
\begin{align*}
    \mb E\exp\bigg\{\bigg(\frac{|XY|}{\lambda_x\lambda_y}\bigg)^\alpha\bigg\} &\leq \mb E\exp\bigg\{\frac{1}{2}\bigg(\frac{|X|}{\lambda_x}\bigg)^{2\alpha} + \frac{1}{2}\bigg(\frac{|Y|}{\lambda_y}\bigg)^{2\alpha}\bigg\}\\
    &=\sqrt{\mb E\exp\bigg\{\bigg(\frac{|X|}{\lambda_x}\bigg)^{2\alpha}\bigg\}}\cdot \sqrt{\mb E\exp\bigg\{\bigg(\frac{|Y|}{\lambda_y}\bigg)^{2\alpha}\bigg\}} \leq 2,
\end{align*}
if we take $\lambda_x = \norm{X}_{\psi_{2\alpha}}$ and $\lambda_y = \norm{Y}_{\psi_{2\alpha}}$. Therefore, by definition
\begin{align*}
    \norm{XY}_{\psi_\alpha} \leq \lambda_x\lambda_y = \norm{X}_{\psi_{2\alpha}}\cdot\norm{Y}_{\psi_{2\alpha}}.
\end{align*}
\end{proof}

\begin{lemma}\label{lem: diagonal_elements}
For every $g\in\m H^\ast$, let
\begin{align*}
    V_N(g) := \frac{1}{N^2}\sum_{i=1}^N\int_0^T\big\langle\xi_{i,i}^g(t),\,\dd\overline{W}_t^{i}\big\rangle.
\end{align*}
Then $V_n(g)$ is sub-Gaussian with parameter $CB^2N^{-3}$ for some positive constant $C$ (may depends on $T$), i.e.,
\begin{align*}
    \barP\Big(V_N(g) > u\Big) \leq e^{-\frac{N^3u^2}{2CB^2}},\qquad\forall\,\, u>0.
\end{align*}
This implies that
\begin{align*}
    \norm{V_N(g)}_{\psi_2} \leq \frac{2B\sqrt{C}}{N\sqrt{N}}.
\end{align*}
\end{lemma}
\begin{proof}
For any integer $p\geq 1$, by Burkholder--Davis--Gundy's inequality,
\begin{align*}
    \mb E_{\barP}\big|V_N(g)\big|^{2p} &\leq \frac{C^pp^p}{N^{4p}}\mb E_{\barP}\bigg(\sum_{i=1}^N\int_0^T\norm{\xi_{i,i}^g(t)}^2\,\dd t\bigg)^p \leq \frac{C^pB^{2p}p^p}{N^{3p}} \leq p!\cdot\frac{C^pB^{2p}}{N^{3p}}.
\end{align*}
The last inequality is by $(p/e)^p \leq p!$. Also, we know $V_N(g)$ is mean-zero. The high probability bound just follows from Appendix in \cite{della2021nonparametric}. The bound of $\psi_2$-norm can be derived by the argument in Lemma \ref{lem: U_statistics}.
\end{proof}

\begin{lemma}[Concentration of degenerate U-statistics]\label{lem: U_statistics}
For every $g\in\m H^\ast$, let $g^\dagger$ be the degenerate kernel defined in {\bf Step 5} in the proof of Theorem~\ref{eqn: converge_of_b} and define the U-statistics
\begin{align*}
    U_N(g) = \frac{1}{N^2}\sum_{1\leq i\neq j\leq N}g^\dagger(\overline{W}^i, \overline{W}^j)
\end{align*}
Then, there is a constant $C>0$ (may depends on $T$) such that
\begin{align*}
    \barP\Big(\big|U_N(g)\big| > u\Big) \leq \exp\Big\{-\Big(\frac{uN}{CB\log\log N}\Big)^{\frac{2}{3}}\Big\}, \qquad \forall\,u\geq \frac{CB(\log N)^2}{N}.
\end{align*}
This implies that
\begin{align*}
    \norm{U_N(g)}_{\psi_{2/3}} \leq \frac{3CB\log\log N}{N}.
\end{align*}
\end{lemma}
\begin{proof}
Let $\{\overline{W}^{il}: 1\leq i\leq N, l=1, 2\}$ be independent copies of standard Brownian motions under $\barP$. By Theorem 3.1.1 in \cite{de1999decoupling},
\begin{align*}
    \big|\!\big|U_N(g)\big|\!\big|_{L_p(\barP)} \lesssim \big|\!\big| U_N^D(g)\big|\!\big|_{L_p(\barP)} := \bigg|\!\bigg|\frac{1}{N^2}\sum_{1\leq i\neq j\leq N}g^\dagger(\overline{W}^{i1},\overline{W}^{j2})\bigg|\!\bigg|_{L_p(\barP)}.
\end{align*}
For simplicity, let
\begin{align*}
    g^\dagger_{ij} = N^{-2}g^\dagger(\overline{W}^{i1}, \overline{W}^{j2}) = \frac{1}{N^2}\int_0^T\big\langle \xi_{i,j}^{g,D}(t),\,\dd\overline{W}_t^{i1}\big\rangle,
\end{align*}
where $\xi_{i,j}^{g,D}$ is the decoupling version of $\xi_{i,j}^g$ defined as
\begin{align*}
    \xi_{i,j}^{g,D}(t) = \tilde{g}(X_t^{i,1}, X_t^{j, 2}) - \int_{\mb T^d}\tilde{g}(X_t^{i,1}, y)\,\dd\mu_t(y).
\end{align*}
Here $X^{i,l}$ is the solution of (\ref{eqn: mean_field_SDE}) corresponding to the Brownian motion $\overline{W}^{i,l}$. For $p\geq 2$, follow the proof of Theorem 3.2 in \cite{gine2000exponential}, we can get
\begin{align}\label{eqn: pmoment_Ustat}
    \mb E_{\barP}\bigg|\sum_{1\leq i\neq j\leq N}g^\dagger_{ij}\bigg|^p
    &\leq C^p\mb E_{\barP}^2\bigg(p^{\frac{p}{2}}\bigg[\sum_{i=1}^N\mb E_{\barP}^1\Big(\sum_{j=1,j\neq i}^Ng^\dagger_{ij}\Big)^2\bigg]^{\frac{p}{2}} + p^p\mb E_{\barP}^1\sum_{i=1}^N\Big|\sum_{j=1,j\neq i}^Ng^\dagger_{ij}\Big|^p\bigg),
\end{align}
by conditioning on $\{\overline{W}^{j,2}: 1\leq j\leq N\}$ first. Applying Proposition 3.1 in \cite{gine2000exponential}, the first term can be bounded by
\begin{align*}
    &p^{\frac{p}{2}}\mb E_{\barP}^2\bigg[\sum_{i=1}^N\mb E_{\barP}^1\Big(\sum_{j=1,j\neq i}^Ng^\dagger_{ij}\Big)^2\bigg]^{\frac{p}{2}}
    \leq p^{\frac{p}{2}}\bigg(\sum_{1\leq i\neq j\leq N}\mb E_{\barP}\big(g^\dagger_{ij}\big)^2\bigg)^{\frac{p}{2}}\\
    &\qquad\qquad\qquad + p^p\mb E \bigg(\sup\bigg\{\mb E_{\barP}^2\sum_{j=1}^N\Big(\mb E_{\barP}^1\sum_{i=1, i\neq j}^Ng^\dagger_{ij} h_i(\overline{W}^{i1})\Big)^2: \mb E_{\barP}^1\sum_{i=1}^Nh_i^2(\overline{W}^{i1}) \leq 1\bigg\}\bigg)^{\frac{p}{2}}\\
    &\qquad\qquad\qquad + p^\frac{3p}{2}\mb E_{\barP}^2\max_{1\leq j\leq N}\Big(\mb E_{\barP}^1\sum_{i=1}^N\big(g^\dagger_{ij}\big)^2\Big)^{\frac{p}{2}}.
\end{align*}
Now, let us bound these three terms. Notice that
\begin{align*}
    \bigg(\sum_{1\leq i\neq j\leq N}\mb E_{\barP}\big(g^\dagger_{ij}\big)^2\bigg)^\frac{p}{2}
    &=\frac{1}{N^{2p}}\bigg(\sum_{1\leq i\neq j\leq N}\mb E_{\barP}\int_0^T\norm{\xi_{i,j}^{g,D}(t)}^2\,\dd t\bigg)^\frac{p}{2}\leq \Big(\frac{CB}{N}\Big)^p.
\end{align*}
Next, we have
\begin{align*}
    &\quad\, \sup\bigg\{\mb E_{\barP}^2\sum_{j=1}^N\Big(\mb E_{\barP}^1\sum_{i=1, i\neq j}^Ng^\dagger_{ij} h_i(\overline{W}^{i1})\Big)^2: \mb E_{\barP}^1\sum_{i=1}^Nh_i^2(\overline{W}^{i1}) \leq 1\bigg\}\\
    &\leq \sup\bigg\{\mb E_{\barP}^2\sum_{j=1}^N\Big(\mb E_{\barP}^1\sum_{i\neq j}\big(g^\dagger_{ij}\big)^2\Big)\Big(\mb E_{\barP}^1\sum_{i\neq j}h_i^2\big(\overline{W}^{i1}\big)\Big): \mb E_{\barP}^1\sum_{i=1}^Nh_i^2(\overline{W}^{i1}) \leq 1\bigg\}\\
    &\leq \mb E_{\barP}^2\sum_{j=1}^N\mb E_{\barP}^1\sum_{i\neq j}\big(g^\dagger_{ij}\big)^2\\
    &\lesssim N^{-2}B^2.
\end{align*}
Lastly,
\begin{align*}
    \mb E_{\barP}^2\max_{1\leq j\leq N}\Big(\mb E_{\barP}^1\sum_{i=1}^N\big(g^\dagger_{ij}\big)^2\Big)^\frac{p}{2}
    &\leq \sum_{j=1}^N\mb E_{\barP}^2\Big(\mb E_{\barP}^1\sum_{i=1}^N\big(g^\dagger_{ij}\big)^2\Big)^\frac{p}{2}\\
    &= \frac{1}{N^{2p}}\sum_{j=1}^N\mb E_{\barP}^2\bigg(\sum_{i=1}^N\mb E_{\barP}^1\int_0^T\norm{\xi_{i,j}^{g,D}(t)}^2\,\dd t\bigg)^\frac{p}{2}\\
    &\leq \frac{C^pB^p}{N^{\frac{3p}{2}-1}}.
\end{align*}

Rather than bound the second term of (\ref{eqn: pmoment_Ustat}) as \cite{gine2000exponential} did, we directly bound it by Burkholder--Davis--Gundy's inequality, see e.g., \cite{barlow1982semi},
\begin{align*}
    p^p\mb E_{\barP}\sum_{i=1}^N\Big|\sum_{j=1,j\neq i}^Ng^\dagger_{ij}\Big|^p
    &=\frac{p^p}{N^{2p}}\sum_{i=1}^N\mb E_{\barP}\bigg|\int_0^T\bigg\langle\sum_{j=1,j\neq i}^N\xi_{i,j}^{g,D},\,\dd \overline{W}^{i1}_t\bigg\rangle\bigg|^p\\
    &\leq \frac{p^p}{N^{2p}}\sum_{i=1}^N C^pp^{\frac{p}{2}}\mb E_{\barP}\bigg(\int_0^T\Big|\!\Big|\sum_{j=1,j\neq i}^N\xi_{i,j}^{g,D}\Big|\!\Big|^2\,\dd t\bigg)^{\frac{p}{2}}\\
    &\leq \bigg(\frac{CBp^{\frac{3}{2}}}{N^{1-\frac{1}{p}}}\bigg)^p.
\end{align*}
Combining all the pieces together, we have
\begin{align*}
    \big|\!\big|U_N(g)\big|\!\big|_{L_p(\barP)}^p &\leq C^pB^p \bigg[\frac{p^{\frac{p}{2}}}{N^p} + \frac{p^p}{N^p} + \frac{p^\frac{3p}{2}}{N^{\frac{3p}{2}-1}} + \frac{p^\frac{3p}{2}}{N^{p-1}}\bigg]\leq C^pB^pp^{3p/2}N^{1-p}.
\end{align*}
By Markov's inequality, we have
\begin{align*}
    \barP\Big(\big|U_N(g)\big| > u\Big) \leq \frac{C^pB^pp^{3p/2}}{u^pN^{p-1}} \leq e^{-p},
\end{align*}
by tuning the parameter $p > 2$ such that $CBp^{3/2} < uN^{1 - 1/p}$. By choosing
\begin{align*}
    p = \bigg(\frac{uN}{eCB\log\log N}\bigg)^{\frac{2}{3}},\qquad \textrm{when}\,\, u\geq \frac{CBe(\log N)^2}{N},
\end{align*}
we get
\begin{align*}
    \barP\Big(\big|U_N(g)\big| > u\Big) \leq \exp\Big\{-\Big(\frac{uN}{CB\log\log N}\Big)^{\frac{2}{3}}\Big\}, \qquad \forall\,u\geq \frac{CB(\log N)^2}{N}
\end{align*}
for some constant $C$.

Lastly, we will prove the statement about the Orlicz norm of $U_N(g)$. Take 
\begin{align*}
    \lambda = \frac{3CB\log\log N}{N},
\end{align*}
and we have
\begin{align*}
    \mb E_{\barP}\exp\bigg\{\bigg(\frac{\big|U_N(g)\big|}{\lambda}\bigg)^\frac{2}{3}\bigg\}
    &= \int_0^\infty \barP\bigg(\big|U_N(g)\big| > \lambda(\log u)^{\frac{3}{2}}\bigg)\,\dd u\\
    &\leq 1 + \int_{1}^\infty \exp\bigg\{-\bigg(\frac{\lambda(\log u)^{3/2}N}{CB\log\log N}\bigg)^{\frac{2}{3}}\bigg\}\,\dd u\\
    &= 1 + \int_{1}^\infty \exp\bigg\{-\bigg(\frac{\lambda N}{CB\log\log N}\bigg)^{\frac{2}{3}}\log u\bigg\}\,\dd u\\
    &= 1 - \bigg[1-\bigg(\frac{\lambda N}{CB\log\log N}\bigg)^{\frac{2}{3}}\bigg]^{-1}\\
    &< 2.
\end{align*}
We finish the proof.
\end{proof}

\begin{lemma}[sub-exponential increments]\label{lem: sub_exp_increments}
Let 
\begin{align*}
    Y_g = \frac{1}{\sqrt{N}}\sum_{i=1}^N\int_0^T\big\langle g(\mu_t, X_t^i),\,\dd W_t^i\big\rangle.
\end{align*}
Then for any $f$, and $g\in\m H^\ast$, $Y_f - Y_g$ is sub-exponential with parameters satisfying
\begin{align*}
    \mb E_{\barP}e^{\lambda(Y_f - Y_g)} \leq \exp\bigg\{\frac{\lambda^2\norm{f-g}_E^2/2}{1 - \sqrt{T/2N}\lambda\norm{f-g}_\infty}\bigg\}.
\end{align*}
Moreover, the Bernstein's bound holds as well, i.e.,
\begin{align*}
    \barP\big(|Y_f - Y_g| > u\big) < 2\exp\bigg\{-\frac{u^2}{2\big(\norm{f-g}_E^2 + \sqrt{T/2N}\norm{f-g}_\infty u\big)}\bigg\}.
\end{align*}
\end{lemma}
\begin{proof}
Note that
\begin{align*}
    &\quad\,\mb E_{\barP}e^{\lambda(Y_f - Y_g)}\\ 
    &= \bigg(\mb E_{\barP}\exp\bigg\{\frac{\lambda}{\sqrt{N}}\int_0^T\big\langle(f-g)(\mu_t, X_t),\,\dd\overline{W}_t\big\rangle\bigg\}\bigg)^N\\
    &\stackrel{\ri}{\leq} \bigg(\mb E_{\barP}\exp\bigg\{\frac{2\lambda}{\sqrt{N}}\int_0^T\big\langle(f-g)(\mu_t, X_t),\,\dd\overline{W}_t\big\rangle - \frac{2\lambda^2}{N}\int_0^T\norm{(f-g)(\mu_t, X_t)}^2\,\dd t\bigg\}\bigg)^{\frac{N}{2}}\\
    &\qquad\cdot\bigg(\mb E_{\barP}\exp\bigg\{\frac{2\lambda^2}{N}\int_0^T\norm{(f-g)(\mu_t, X_t)}^2\,\dd t\bigg\}\bigg)^{\frac{N}{2}}\\
    &\stackrel{\rii}{=} \bigg(\mb E_{\barP}\exp\bigg\{\frac{2\lambda^2}{N}\int_0^T\norm{(f-g)(\mu_t,X_t)}^2\,\dd t\bigg\}\bigg)^{\frac{N}{2}}.
\end{align*}
Here, (i) is by Cauchy--Schwartz's inequality, and (ii) is by the fact that
\begin{align*}
    \exp\bigg\{\frac{2\lambda}{\sqrt{N}}\int_0^s\big\langle(f-g)(\mu_t, X_t),\,\dd\overline{W}_t\big\rangle - \frac{1}{2}\Big(\frac{2\lambda}{\sqrt{N}}\Big)^2\int_0^s\norm{(f-g)(\mu_t,X_t)}^2\,\dd t\bigg\},\quad 0\leq s\leq T
\end{align*}
is an exponential martingale under $\barP$. Now, by Taylor's formula
\begin{align*}
    &\quad\,\mb E_{\barP}\exp\bigg\{\frac{2\lambda^2}{N}\int_0^T\norm{(f-g)(\mu_t, X_t)}^2\,\dd t\bigg\}\\
    &= 1 + \sum_{k=1}^\infty\frac{1}{k!}\Big(\frac{2\lambda^2}{N}\Big)^k\mb E_{\barP}\bigg(\int_0^T\norm{(f-g)(\mu_t, X_t)}^2\,\dd t\bigg)^k\\
    &\leq 1 + \sum_{k=1}^\infty\Big(\frac{2\lambda^2}{N}\Big)^k\big(T\norm{f-g}_\infty^2\big)^{k-1}\norm{f-g}_E^2\\
    &= 1 + \frac{2\lambda^2\norm{f-g}_E^2/N}{1 - 2T\lambda^2\norm{f-g}_\infty^2}\\
    &\leq \exp\bigg\{\frac{2\lambda^2\norm{f-g}_E^2/N}{1-2T\lambda^2\norm{f-g}_\infty^2/N}\bigg\}\\
    &\leq\exp\bigg\{\frac{2\lambda^2\norm{f-g}_E^2/N}{1-\sqrt{2T/N}\lambda\norm{f-g}_\infty}\bigg\}
\end{align*}
for $N$ large enough. So
\begin{align*}
    \mb E_{\barP}e^{\lambda(Y_f - Y_g)} \leq \exp\bigg\{\frac{\lambda^2\norm{f-g}_E^2/2}{1 - \sqrt{T/2N}\lambda\norm{f-g}_\infty}\bigg\}.
\end{align*}
Then the Bernstein's bound holds by Proposition 2.10 in \cite{wainwright2019high}.
\end{proof}

\begin{lemma}[estimation of $\psi_1$-norm]\label{lem: psi1_chaining}
\begin{align*}
    \bigg|\!\bigg|\sup_{g\in\m H^\ast_u}|Y_g|\bigg|\!\bigg|_{\psi_1} \lesssim \sqrt{\frac{TB^2}{2N}}\int_0^\frac{1}{2}\log\big(1 + H(\varepsilon, \m H^\ast_u)\big)\,\dd \varepsilon + u\int_0^\frac{1}{2}\sqrt{\log\big(1 + H(\varepsilon,\m H^\ast_u)\big)}\,\dd \varepsilon.
\end{align*}
\end{lemma}
\begin{proof}
We only need to consider the case where $\m H^\ast_u$ is a finite set, and the whole proof can be extended to a separable $\m H^\ast_u$ ($\m H^\ast$ is separable indeed) by a standard argument. Let $S_k\subset \m H^\ast_u$ such that $|S_k| = H(2^{-k}, \m H^\ast_u)$ for all integer $k\geq 0$. We specify $S_0 = \{0\}$ and $S_K = \m H^\ast_u$. Such $K$ exists since $\m H^\ast_u$ is finite. Let $\pi_k(g)$ be an element in $S_k$, such that
\begin{align*}
    \norm{g - \pi_k(g)}_E \leq 2^{-k}u,\qquad \norm{g - \pi_k(g)}_\infty \leq 2^{-k}B.
\end{align*}
For any $g\in\m H^\ast_u$ and $1\leq k\leq K$, let $g^K = g$ and $g^{k-1} = \pi_{k-1}(g^k)$. Then, we can decompose $\sup_{g\in\m H^\ast_u}Y_g$ by
\begin{align*}
    \sup_{g\in\m H^\ast_u} |Y_g| = \sup_{g\in\m H^\ast_u}|Y_g - Y_0| \leq \sum_{k=1}^K \sup_{g\in\m H^\ast_u} \big|Y_{g^k} - Y_{g^{k-1}}\big| \leq \sum_{k=1}^K\sup_{g\in S_k}\big|Y_g - Y_{\pi_{k-1}(g)}\big|.
\end{align*}
By Lemma 8.3 in \cite{kosorok2008introduction}, Lemma \ref{lem: sub_exp_increments}, and triangular inequality of $\psi_1$-norm, we can get
\begin{align*}
    \bigg|\!\bigg|\sup_{g\in\m H^\ast_u}|Y_g|\bigg|\!\bigg|_{\psi_1} &\leq \sum_{k=1}^K\bigg|\!\bigg|\sup_{g\in S_k}\big|Y_g - Y_{\pi_{k-1}(g)}\big|\bigg|\!\bigg|_{\psi_1}\\
    &\lesssim \sum_{k=1}^K\bigg[\sqrt{\frac{T}{2N}}2^{-k}B\cdot\log\big(1 + H(2^{-k}, \m H^\ast_u)\big) + 2^{-k}u\cdot\sqrt{\log\big(1 + H(2^{-k}, \m H^\ast_u)\big)}\bigg]\\
    &\lesssim \sqrt{\frac{TB^2}{2N}}\int_0^\frac{1}{2}\log\big(1 + H(\varepsilon, \m H^\ast_u)\big)\,\dd \varepsilon + u\int_0^\frac{1}{2}\sqrt{\log\big(1 + H(\varepsilon,\m H^\ast_u)\big)}\,\dd \varepsilon.
\end{align*}
\end{proof}

\section{Proof of Applications}
\begin{proof}[proof of Corollary~\ref{coro: holder_smooth}]
First, let us calculate the order of $r_N$. Notice that if $\tilde{f}$, $\tilde{g}\in\tilde{\m H}^\ast$, s.t. $\big|\!\big|\tilde{f} - \tilde{g}\big|\!\big|_\infty < \varepsilon$, we have $\norm{f-g}_\infty < \varepsilon$. This implies $\log N(\varepsilon, \m H^\ast, \norm{\cdot}_\infty) \leq 2\log N(\varepsilon/2, \tilde{\m H^\ast}, \norm{\cdot}_\infty)$. Therefore
\begin{align*}
    J_5(r_N) \lesssim \int_0^{\frac{r_N}{2}} \sqrt{\varepsilon^{-\frac{d}{\alpha}}}\,\dd\varepsilon \lesssim r_N^{1-\frac{d}{2\alpha}}.
\end{align*}
By letting $r_N^{1-\frac{d}{2\alpha}} \lesssim \sqrt{N}r_N^2$, we get $r_N \lesssim N^{-\frac{\alpha}{d+2\alpha}}$. Next, we shall calculate the order of $\delta_N$. In fact,
\begin{align*}
    &\quad\,\mb E_{\barP}\sup_{g\in\m H^\ast_u}\bigg|\frac{1}{N}\sum_{i=1}^N\int_0^T\big\langle g(\mu_t, X_t^i),\,\dd\overline{W}_t^i\big\rangle\bigg| \\
    &\lesssim \norm{\sup_{g\in\m H^\ast_u}\bigg|\frac{1}{N}\sum_{i=1}^N\int_0^T\big\langle g(\mu_t, X_t^i),\,\dd\overline{W}_t^i\big\rangle\bigg|}_{\psi_1}\\
    &\lesssim \sqrt{\frac{TB^2}{2N^2}}\int_0^\frac{1}{2}\log\big(1 + H(\varepsilon, \m H^\ast_u)\big)\,\dd \varepsilon + \frac{u}{\sqrt{N}}\int_0^\frac{1}{2}\sqrt{\log\big(1 + H(\varepsilon,\m H^\ast_u)\big)}\,\dd \varepsilon.
\end{align*}
The last inequality is by Lemma \ref{lem: psi1_chaining}. Note that 
\begin{align*}
    \norm{f-g}_E^2 = \int_0^T\!\!\int_{\mb T^d}\norm{(f-g)(\mu_t, x)}^2\,\dd \mu_t(x)\dd t \leq T\big|\!\big|\tilde{f}-\tilde{g}\big|\!\big|_\infty^2,
\end{align*}
and
\begin{align*}
    \norm{f-g}_\infty \leq \big|\!\big|\tilde{f} - \tilde{g}\big|\!\big|_\infty.
\end{align*}
Therefore $\log H(\varepsilon, \m H^\ast_u) \leq 2\log N\big((2\sqrt{T}+2)^{-1}\varepsilon u, \tilde{\m H}^\ast, \norm{\cdot}_\infty\big) \lesssim (\varepsilon u)^{-d/\alpha}$. So we know
\begin{align*}
    \mb E_{\barP}\sup_{g\in\m H^\ast_u}\bigg|\frac{1}{N}\sum_{i=1}^N\int_0^T\big\langle g(\mu_t, X_t^i),\,\dd\overline{W}_t^i\big\rangle\bigg|
    &\lesssim \frac{1}{N}\int_0^{\frac{1}{2}}(\varepsilon u)^{-\frac{d}{\alpha}}\,\dd\varepsilon + \frac{u}{\sqrt{N}}\int_0^{\frac{1}{2}}(\varepsilon u)^{-\frac{d}{2\alpha}}\,\dd\varepsilon\\
    &\lesssim \frac{u^{-\frac{d}{\alpha}}}{N} + \frac{u^{1-\frac{d}{2\alpha}}}{\sqrt{N}}.
\end{align*}
By Letting
\begin{align*}
    \frac{\delta_N^{-\frac{d}{\alpha}}}{N} + \frac{\delta_N^{1-\frac{d}{2\alpha}}}{\sqrt{N}} \lesssim \delta_N^2
\end{align*}
we get $\delta_N\lesssim N^{-\frac{\alpha}{d+2\alpha}}$. By Theorem \ref{thm: converge_of_b}, we know $\big|\!\big|\wht b_N - b^\ast\big|\!\big|_E \lesssim N^{-\frac{\alpha}{d+2\alpha}}$ with probability $1 - C\exp\big\{-C'(\frac{\log N}{\log\log N})^{2/3}\big\}$ for some positive constant $C$ and $C'$ that may depends on $T$.
\end{proof}

\begin{lemma}\label{lem: stability}
Under Assumption \ref{assump: non-stationary} for all $C > 0$,
\begin{align*}
    \norm{\wht F_N - F^\ast}_2 \leq \frac{\norm{\m L(\wht b_N - b^\ast)(\mu, \cdot)}_2}{\inf_{0 < \norm{k}< C} \|(\m L\mu)_k\|} + \frac{\|\wht F_N - F^\ast\|_{H^1}}{2\pi C}.
\end{align*}
\end{lemma}
\begin{proof}[Proof of Lemma \ref{lem: stability}]
By Plancherel's identity (Proposition 3.1.16 in \cite{grafakos2008classical})
\begin{align*}
    \norm{\wht F_N - F^\ast}_2^2 &= \sum_{ k\in\mb Z^d} \norm{(\wht F_N - F^\ast)_k}^2\\
    &= \sum_{0 < \norm{k} < C}\norm{(\wht F_N - F^\ast)_k}^2 + \sum_{\norm{k} \geq C}\norm{(\wht F_N - F^\ast)_k}^2.
\end{align*}
The last equality is by the restriction $\int_{\mb T^d}\wht F_N(x)\,\dd x = \int_{\mb T^d}F^\ast (x)\,\dd x = 0$, which implies that $(\wht F_N - F^\ast)_0 = 0$. By~\eqref{eqn: deconvolution} and Assumption \ref{assump: non-stationary} that $(\m L\mu)_k \neq 0$ for all $k\neq 0$,
\begin{align*}
    \sum_{0 < \norm{k} < C}\norm{(\wht F_N - F^\ast)_k}^2 &= \sum_{0< \norm{k} < C} \norm{\frac{\big(\m L(\wht b_N - b^\ast)(\cdot, \mu)\big)_k}{(\m L\mu)_k}}^2\\
    &\leq \inf_{0<\norm{k}<C}|(\m L\mu)_k|^{-2}\cdot \sum_{k\in\mb Z^d}\norm{\big(\m L(\wht b_N - b^\ast)(\cdot, \mu)\big)_k}^2\\
    &= \inf_{0<\norm{k}<C}|(\m L\mu)_k|^{-2}\cdot\norm{\m L(\wht b_N - b^\ast)(\cdot, \mu)}_2^2.
\end{align*}
For the second term, recall that for any function $f: \mb T^d\to\mb R^d$ and $k\in\mb Z^d$
\begin{align*}
    \sum_{l=1}^d \norm{(\nabla f_l)_k}^2 &= \sum_{l,j=1}^d \Big|\Big(\frac{\partial f_l}{\partial x_j}\Big)_k\Big|^2 = \sum_{l,j=1}^d\big|2\pi ik_j(f_l)_k\big|^2\\
    &= 4\pi^2\sum_{l=1}^d \big|(f_l)_k\big|^2 \sum_{j=1}^dk_j^2\\
    &=4\pi^2\norm{k}^2\norm{(f)_k}^2.
\end{align*}
Take $f = \wht F_N - F^\ast$, and we get
\begin{align*}
    \sum_{\norm{k} \geq C}\norm{(\wht F_N - F^\ast)_k}^2 &\leq \frac{1}{(2\pi C)^{2}}\cdot \sum_{\norm{k}>C} (2\pi)^{2} |k|^{2} \norm{(\wht F_N - F^\ast)_k}^2\\
    &= \frac{1}{(2\pi C)^{2}}\sum_{\norm{k}>C} \sum_{l=1}^d\norm{\big(D^j(\wht F_N - F^\ast)\big)_k}^2\\
    &\leq \frac{1}{(2\pi C)^{2}}\sum_{l=1}^d \sum_{k\in\mb Z^d}\norm{\big(\nabla(\wht F_N - F^\ast)_l\big)_k}^2\\
    &= \frac{1}{(2\pi C)^{2}}\sum_{l=1}^d\norm{\nabla(\wht F_N - F^\ast)_l}_2^2\\
    &\leq (2\pi C)^{-2}\norm{\wht F_N - F^\ast}_{H^1}^2.
\end{align*}
Combining these pieces together yields the result.
\end{proof}

\begin{proof}[Proof of Corollary \ref{cor: rate_of_interaction}]
We can assume $\mu_0(x) > 0$ for all $x$. Otherwise, we can choose $w(t)$ and $\rho(t)$ satisfying $\int_s^Tw(t)\,\dd\rho(t) = 0$ for some time $s > 0$, and apply the operator $\m L$ with new $w$ and $\rho$.
\begin{align*}
    \big|\!\big|\m L(\wht b_N - b^\ast)(\mu, \cdot)\big|\!\big|_2^2
    &= \int_{\mb T^d}\bigg|\!\bigg| \int_0^T w(t) (\wht b_N - b^\ast)(\mu_t, x)\,\dd t\bigg|\!\bigg|^2\,\dd x\\
    &\leq \int_{\mb T^d}\!\int_0^T w^2(t)\big|\!\big|(\wht b_N - b^\ast)(\mu_t, x)\big|\!\big|^2\,\dd t\,\dd x\\
    &\lesssim \int_0^T\!\!\int_{\mb T^d}\big|\!\big|(\wht b_N - b^\ast)(\mu_t, x)\big|\!\big|^2\,\dd x\,\dd t\\
    &\lesssim \int_0^T\!\!\int_{\mb T^d}\big|\!\big|(\wht b_N - b^\ast)(\mu_t, x)\big|\!\big|^2\,\dd\mu_t(x)\,\dd t.
\end{align*}
The last inequality is by the fact that $\mu_t(x)$ is bounded away from zero, since $[0, T]\times \mb T^d$ is compact. Then by Lemma \ref{lem: stability}, with high probability we have
\begin{align*}
    \big|\!\big|\wht F_N - F^\ast\big|\!\big|_2 \leq \frac{\delta_N + r_N + \log N/ N}{\inf_{0<\norm{k}<C}\norm{(\m L\mu)_k}} + \frac{1}{\pi C}\sup_{F\in \tilde{\m H}}\norm{F}_{H^1}.
\end{align*}
By taking $C = \eta_N$, we finish the proof.
\end{proof}
\end{document}